\documentclass[12pt]{amsart}

\usepackage{lineno,hyperref}
\usepackage{amsmath,amsfonts,amsthm}
\usepackage{amssymb,latexsym}
\usepackage{cases}
\usepackage[numbers]{natbib}
\usepackage{booktabs}
\usepackage{bm} 
\usepackage{cleveref}
\usepackage{paralist}
\usepackage{multirow}
\usepackage{mfirstuc}

\renewcommand\appendix{\setcounter{secnumdepth}{-2}}

\usepackage{url,cancel}
\usepackage{float}

\usepackage{caption,enumitem}
\newtheorem{thm}{Theorem}[section]

\newtheorem{lem}[thm]{Lemma}
\newtheorem{prop}[thm]{Proposition}
\newtheorem{cor}[thm]{Corollary}

\theoremstyle{remark}
\newtheorem{re}[thm]{Remark}

\theoremstyle{definition}

\newtheorem{defn}[thm]{Definition}
\newcommand{\ord}{\mathrm{ord}}
\newcommand{\Z}{{\mathbb Z}}
\newcommand{\Q}{{\mathbb Q}}

\newcommand{\cO}{{\mathcal O}}

\newcommand{\p}{{\mathfrak p}}
\newcommand{\pr}{\mathrm{pr}}

\def\rep{{\to\!\!\! -}}

\makeatletter

\def\firstletterparse#1#2&{\def\strfirstletter{#1}\def\strotherletters{#2}}

\newcommand{\MakeSmallcaps}[1]{%
	\expandafter\firstletterparse#1&
	\expandafter\MakeUppercase\strfirstletter\textsc{\strotherletters}%
}

\@addtoreset{equation}{section}
\makeatother
\numberwithin{equation}{section}

\begin{document}

	\author{Zilong He}
	\address{School of Computer Science and Technology, Dongguan University of Technology, Dongguan 523808, P.R. China}
	\email{zilonghe@connect.hku.hk}
	\author{Yong Hu}
	\address{Department of Mathematics,	Southern University of Science and Technology, Shenzhen 518055, P.R. China}
	\email{huy@sustech.edu.cn}
	
	\title[ ]{On $n$-universal quadratic forms over dyadic local fields}
	\thanks{ }
	\date{\today}
	\begin{abstract}
		Let $ n \ge 2$ be an integer. We give necessary and sufficient conditions for an integral quadratic form over dyadic local fields to be $ n $-universal by using invariants from Beli's theory of bases of norm generators. Also, we provide a minimal set for testing $ n $-universal quadratic forms over dyadic local fields, as an  analogue of Bhargava and Hanke's 290-theorem (or Conway and Schneeberger's 15-theorem) on universal quadratic forms with integer coefficients.
	\end{abstract}
	\maketitle
	
	\section{Introduction}
	
	The term \emph{universal quadratic form} was coined by Dickson \cite{Dic29a} for indefinite quadratic forms over $\Z$ and extended to the positive definite case by Ross \cite{Ros32}. It means that the quadratic form under consideration represents all integers, or all positive integers if it is positive definite. Extending Ramanujan's work \cite{ramanujan_on_1917} in the case of diagonal quaternary forms, Dickson and his students made important contributions to the classification of universal quadratic forms over $\Z$ (see e.g. \cite{Dic27a}, \cite{Dic27b}, \cite{Dic29a}, \cite{Dic29b}, \cite{Ros32} and  \cite{will_determinatino_1948}). In 1993, Conway and Schneeberger proved a simple criterion for universality of classic forms (i.e. quadratic forms with integer matrix). Their theorem is now called the 15-theorem (see \cite{conway_15-theorem-2000}) because it shows that a positive definite classic quadratic form over $\Z$ is universal if and only if it represents every positive integer up to  $15$. The analogous result for arbitrary positive definite integral quadratic forms, known as the 290-theorem, is  proved later by Bhargava and Hanke \cite{bhargava_290-theorem-2005}.

	For any positive integer $n$, B. M. Kim, M.-H. Kim and S. Raghavan \cite{kim_2universal_1997} defined a positive definite classic quadratic form over $\Z$ to be $n$-universal if it represents all $n$-ary classic forms. With this definition, two theorems due to Mordell \cite{mordell_waring_1930} and Ko \cite{ko_on_1937} may be rephrased as asserting that the sum of $n+3$ squares is $n$-universal for $2\le n\le 5$. B. M. Kim, M.-H. Kim and B.-K. Oh \cite{kim_2universal_1999} completed the classification of $ 2 $-universal quinary classic quadratic forms and provided a criterion for $ 2 $-universality of classic quadratic forms analogous to the Conway-Schneeberger theorem. B.-K. Oh \cite{oh_universal_2000} further determined the minimal rank of $ n $-universal classic forms and found all $ n $-universal classic forms over $\Z$ of minimal rank for $ 6\le n\le 8 $.
	
	Representations of quadratic forms can be considered more generally over the ring of integers of a general number field or local field. In the recent papers \cite{xu_indefinite_2020} and \cite{hhx_indefinite_2021}, number fields over which the local-global principle for $n$-universality of indefinite quadratic forms holds are completely determined. A key step in the proofs has been a complete determination of  $n$-universal forms over non-dyadic local fields and some partial results in the dyadic case. For $n=1$, Beli's work  \cite{beli_universal_2020}   complements the analysis over dyadic fields in \cite[\S\,2]{xu_indefinite_2020}, and gives necessary and sufficient conditions for an integral quadratic form over a general dyadic field to be universal. His method builds upon the general theory of \emph{bases of norm generators} (BONGs), which he developed in his thesis \cite{beli_thesis_2001} (see also \cite{beli_integral_2003}, \cite{beli_representations_2006}, \cite{beli_Anew_2010}, \cite{beli_representations_2019}). Without using BONGs, the authors determined in \cite{HeHu1} all integral $n$-universal forms over any  unramified dyadic local field.
	
	\medskip
	
	In this paper, we prove necessary and sufficient conditions characterizing $n$-universal integral quadratic forms  over a general dyadic local field (Theorem \ref{thm:nuniversaldyadic}). Unlike in the other work \cite{HeHu1}, here we have to use Beli's theory of BONGs and our results are stated in terms of the invariants associated to BONGs. Due to the complexity of Jordan splitting structures, the representation theory of integral quadratic forms over general dyadic fields had remained uncompleted until Beli's work. So we feel that it will be right to use BONGs to obtain  the results in this paper.\footnote{Also based on the BONG theory, some useful results having close relations with ours have been obtained in section 4 of \cite{beli_universal_2020}, which had been updated after the first version of our preprint was available on \texttt{arXiv.org}.}
	
	In addition to the equivalent conditions, we also prove a Bhargava--Hanke (or Conway--Schneeberger) type theorem  (Theorem \ref{thm:nuniversaldyadic15theorem}). Namely, we provide a finite set of $n$-ary  forms such that an integral quadratic form is $n$-universal if and only if it represents all forms in that set, and we show that the set given is minimal for the $n$-universal property test. Indeed, lattices in the testing set are expressed explicitly in terms of Jordan splittings.
	
	\medskip
	
	If one only considers representations of classic forms, there is also the notion of \emph{classic $n$-universal forms} (see e.g. \cite[Definition\;1.4]{hhx_indefinite_2021}). In the unramified case, these forms have been classified in \cite{HeHu1}. For general dyadic fields, this is done in \cite{He22} by the first named author.
	
	\medskip
	
	\noindent {\bf Notation and terminology.} Throughout the paper, let $F$ be a fixed dyadic local field, i.e. a finite extension of the field $\Q_2$ of 2-adic numbers. Let $ \mathcal{O}_{F} $ be the ring of integers (or the valuation ring) of $ F $ and let $ \mathcal{O}_{F}^{\times}$ be its group of units. We write $ \mathfrak{p} $ for the unique maximal ideal of $\mathcal{O}_F$ and $ \pi \in\mathfrak{p}$ for a fixed prime element. Let $\ord: F\to \Z\cup\{\infty\}$ denote the normalized discrete valuation of $F$ and put $ e:=\ord(2)$. For a fractional or zero ideal $\mathfrak{a}$ of $F$, put  $\ord(\mathfrak{a})=\min\{\ord(\alpha)\,|\,\alpha\in \mathfrak{a}\}$.
	
	For any $a,b\in F^{\times}:=F\setminus\{0\}$, let $(a,b)_{\mathfrak{p}}$ denote the Hilbert symbol. For any $c\in F^{\times}$,  its \textit{quadratic defect} is defined by $ \mathfrak{d}(c):=\bigcap_{x\in F}(c-x^{2})\mathcal{O}_{F}$. The \textit{order of relative quadratic defect} is the function
	\[
	d\,:\; F^{\times}/F^{\times 2}\longrightarrow \mathbb{N}\cup \{\infty\}\,;\quad  c\longmapsto d(c):=\ord (c^{-1}\mathfrak{d}(c))\,.
	\]We fix a unit $\Delta:=1-4\rho $ with $\mathfrak{d}(\Delta)=4\mathcal{O}_{F} $ and $\rho\in\mathcal{O}_{F}^{\times}$ (cf. \cite[\S\;93, p.\hskip 0.1cm 251]{omeara_quadratic_1963}). Recall the following properties of the function $ d $ (cf. \cite[\S 1]{beli_integral_2003}, \cite[\S 1.1]{beli_universal_2020}):
	
	(1) The image of $d$ is $\{0,\,1,\,3,\,\cdots, 2e-1,\,2e,\,\infty\}$, and we have $ d(c)=\infty$ if and only if $c\in F^{\times 2} $, $ d(c)=2e$ if and only if $c\in \Delta F^{\times 2} $, and $ d(c)=0$ if and only if $\ord(c) $ is odd.
	
	(2) Domination principle: $ d(ab)\ge \min\{d(a),d(b)\} $.
	
	(3) If $ d(a)+d(b)>2e $, then $ (a,b)_{\mathfrak{p}}=1 $.

	To study quadratic forms, we adopt the geometric language of quadratic spaces and lattices from \cite{omeara_quadratic_1963}.  Unless otherwise stated, quadratic spaces and lattices in this paper are all assumed nondegenerate. The quadratic map associated to a quadratic space or lattice is usually denoted by $Q$. We call an $ \mathcal{O}_{F} $-lattice $ L $ \textit{integral} if $ Q(L):=\{Q(x):x\in L\}\subseteq \mathcal{O}_{F}$, or equivalently, if its \emph{norm} $ \mathfrak{n}L:=Q(L)\mathcal{O}_{F}$ is contained in $\mathcal{O}_F$.

	When a lattice $K$ is represented by another lattice $L$ (in the sense of \cite[p.220]{omeara_quadratic_1963}), we write $K\rep L$. Similarly for quadratic spaces. For a positive integer $n$, an integral $\mathcal{O}_{F}$-lattice  is called \textit{$ n $-universal} if it represents all integral $ \mathcal{O}_{F} $-lattices of rank $ n $. Similarly, a quadratic space over $ F $ is called \textit{$ n $-universal} if it represents all quadratic spaces of dimension $ n $ over $ F $.
	
	We write $ V\cong [a_{1},\ldots,a_{n}] $ (resp. $ L\cong \langle a_{1},\ldots,a_{n}\rangle $) if $ V=Fx_{1}\perp \ldots \perp Fx_{n} $ (resp. $ L=\mathcal{O}_{F}x_{1}\perp \ldots\perp \mathcal{O}_{F}x_{n} $) with  $ Q(x_{i})=a_{i} $. Following Beli's notation,  if $ x_{1},\ldots, x_{n} $ is a BONG for $L$ (cf. Definition\;\ref{defn:BONG})  with  $ Q(x_{i})=a_{i} $, we write $ L\cong \prec a_{1},\ldots, a_{n} \succ $.
	
	For any $\gamma\in F^{\times} $ and $ \xi, \eta\in F $, let $ \gamma A(\xi,\eta)$ denote the binary lattice represented by the matrix $\begin{pmatrix}
		\gamma\xi  &  \gamma  \\
		\gamma  & \gamma\eta
	\end{pmatrix} $ as in \cite[p.255]{omeara_quadratic_1963}. We also write $\mathbf{H}=2^{-1}A(0,\,0)$. The quadratic space it spans is the hyperbolic plane, denoted by $\mathbb{H}$. For any $n\in\mathbb{N}$, let $ \mathbb{H}^{n} $ and $ \mathbf{H}^{n}$ denote the orthogonal sum of $ n $ copies of $\mathbb{H} $ and $ \mathbf{H} $ respectively.
	
	For $ h,k\in\mathbb{Z} $, we  write $ [h,k]^{E} $ (resp. $ [h,k]^{O} $) for the set of all even (resp. odd) integers $ i $ such that $ h\le i\le k$.
	
	\medskip

	Our first main result is the following criterion for $ n $-universality.
	\begin{thm}\label{thm:nuniversaldyadic}
		Let  $ n\ge 2 $ be an integer and let $M$ be  an integral $ \mathcal{O}_{F} $-lattice. Suppose that  $M\cong \prec a_{1},\ldots,a_{m}\succ$  relative to some good BONG and put $ R_{i}=\ord(a_{i})$ for $ 1\le i\le m $.
		
		Then $ M $ is n-universal if and only if either
		\[
		m=n+2=4\,,\;  FM\cong \mathbb{H}^{2} \quad\text{ and }\quad R_{1}=R_{3}=R_{2}+2e=R_{4}+2e=0\,,
		\] or $ m\ge n+3 $ and the following conditions hold:
		
		\begin{itemize}
			\item[(i)] $ R_{i}=0 $ for $ i\in [1,n]^{O} $ and $ R_{i}=-2e $ for $ i\in [1,n]^{E} $.
			\item[(ii)] In case $ n $ is even, one has $ R_{n+1}=0 $ and the following conditions hold:
			
			\begin{enumerate}
				\item[(1)]  $R_{n+2}\in  [-2e,0]^{E}\cup\{1\} $;  and if $ R_{n+2}\in [2-2e,0]^{E} $, then the following conditions hold:
				\begin{enumerate}
					\item[(a)]  If $ R_{n+2}=2-2e $, then $ d(-a_{n+1}a_{n+2})=2e-1 $ or $ R_{n+3}\in \{0,1\} $.
					
					\item[(b)] If $ R_{n+2}\not=2-2e $, then $ d(-a_{j}a_{j+1})=1-R_{j+1} $ for some $ n+1\le j\le m-1 $.
				\end{enumerate}
				\item[(2)]  If $ R_{n+3}-R_{n+2}>2e $, then $ R_{n+2}=-2e $; and if moreover $n\ge 4$, or $ n=2 $ and $ d(a_{1}a_{2}a_{3}a_{4})=2e$, then $R_{n+3}=1$.	
			\end{enumerate}
			
			\item[(iii)] In case $ n $ is odd, one has:

			\begin{enumerate}
				\item[(1)]  $R_{n+1}\in [-2e,0]^{E}\cup \{1\} $; and if $ R_{n+1}\in [4-2e,0]^{E} $, then $ d(-a_{j}a_{j+1})=1-R_{j+1} $ for some $ n\le j\le m-1 $.
 
				\item[(2)]  Suppose $ R_{n+1}=1 $, or  $ R_{n+1}\not=-2e $ and $ R_{n+2}>1 $.
				
				\begin{enumerate}
					\item[(a)]  If $ R_{n+2}-R_{n+1} $ is even, then $ R_{n+3}+R_{n+2}-2R_{n+1}\le 2e-2 $ or
					
					 $ d(-a_{j}a_{j+1})\le 2e+R_{n+1}-R_{j+1}-1 $ for some $ n+2\le j\le m-1$.
					
					\item[(b)] If $ R_{n+2}-R_{n+1} $ is odd, then $ R_{n+3}+R_{n+2}-2R_{n+1}\le 2e $ or
					
					 $ d(-a_{j}a_{j+1})\le 2e+R_{n+1}-R_{j+1} $ for some $ n+2\le j\le m-1$.
				\end{enumerate}
				\item[(3)] If $ R_{n+1}=-2e $, then $ R_{n+2}\in\{0,1\}$.
				
				\item[(4)] $ R_{n+3}-R_{n+2}\le 2e $.	
			\end{enumerate}
		\end{itemize}	
	\end{thm}
	The proof of  Theorem \ref{thm:nuniversaldyadic} will be given in Section \ref{sec:proof-main}.
	The criterion given in this theorem is effective in the sense that in practice we do have a method to find a good BONG for any integral lattice (cf. \cite[\S 7, p.\hskip 0.1cm 109]{beli_representations_2006}).
	
	Just as Beli's result on 1-universality (\cite[Theorem 2.1]{beli_universal_2020}), Theorem \ref{thm:nuniversaldyadic} can be stated in a more compact form if some more notations of Beli are used (see Theorem \ref{thm:even-nuniversaldyadic} for even $ n $ and  Theorem  \ref{cor:odd-nuniversaldyadic-simplify} for odd $ n$).
	
	For even $n$, a simplified version of Theorem\;\ref{thm:nuniversaldyadic} will be given in Theorem\;\ref{thm:2universaldyadicver2}. Corollary \ref{cor:even-nuniversaldyadic-quaternary} will show that our result on quaternary 2-universal lattices agrees with \cite[Proposition\;4.5]{hhx_indefinite_2021}.
	
	\medskip
	
	In the course of proving Theorem \ref{thm:nuniversaldyadic} we also obtain a local analogue of Bhargava and Hanke's 290-theorem. Here, let us call a set of rank $n$ lattices a \textit{testing set} for $ n $-universality if every integral lattice representing all lattices in the set is $n$-universal. A testing set is said to be \textit{minimal} if none of its proper subsets is sufficient for testing $n$-universality.

	In \cite[Corollary\;4.3]{hhx_indefinite_2021}, a testing set for 2-universality is obtained with only classical methods (see also \cite[Proposition\;3.2]{hhx_indefinite_2021} in the non-dyadic case). Using the BONG theory, we determine for general $n\ge 2$, the $ \mathcal{O}_{F} $-maximal lattices on all $ n $-dimensional quadratic spaces and show that they form a minimal testing set for $ n $-universality. In Theorem\;\ref{thm:nuniversaldyadic15theorem} below,  we describe these $ \mathcal{O}_{F} $-maximal lattices explicitly in terms of minimal norm splittings in the sense of \cite{xu_minimal_2003} (see also Proposition \ref{prop:maximallattices}).

	\begin{thm}\label{thm:nuniversaldyadic15theorem}
		Let $n\ge 2$ and let $\mathcal{U}$ be a complete system of representatives of $\mathcal{O}_F^{\times}/\mathcal{O}_F^{\times 2}$ such that $d(\delta)=\ord(\delta-1)$ for all $\delta\in \mathcal{U}$.
		
		\begin{enumerate}
			\item[(i)] If $n$ is even, a minimal testing set for $n$-universality consists of the following lattices:
		\[
\begin{split}
				&\mathbf{H}^{\frac{n}{2}},\;\;  \mathbf{H}^{\frac{n-4}{2}}\perp 2^{-1}A(2,2\rho)\perp 2^{-1}\pi A(2,2\rho)(\text{if }\;  n\ge 4), \\
 &\mathbf{H}^{\frac{n-2}{2}}\perp 2^{-1}A(2,2\rho),\;\;	 \mathbf{H}^{\frac{n-2}{2}}\perp 2^{-1}\pi A(2,2\rho), \\
				&\mathbf{H}^{\frac{n-2}{2}}\perp   \pi^{\frac{1-d(\delta)}{2}}A\big(\pi^{\frac{d(\delta)-1}{2}},-(\delta-1)\pi^{\frac{1-d(\delta)}{2}}\big),\\
			  &\mathbf{H}^{\frac{n-2}{2}}\perp (1+4\rho(\delta-1)^{-1}) \pi^{\frac{1-d(\delta)}{2}}A\big(\pi^{\frac{d(\delta)-1}{2}},-(\delta-1)\pi^{\frac{1-d(\delta)}{2}}\big),\\
&\mathbf{H}^{\frac{n-2}{2}}\perp \langle 1,-\varepsilon\pi \rangle, \quad  \mathbf{H}^{\frac{n-2}{2}}\perp  \langle \Delta,-\Delta\varepsilon\pi \rangle
\end{split}	\]
 for all $\varepsilon\in \mathcal{U}$ and all $\delta\in \mathcal{U}\backslash \{1,\Delta\}$.
			
			\item[(ii)] If $n$ is odd, a minimal testing set for $n$-universality consists of the following lattices:
			\[\begin{split}
			&\mathbf{H}^{\frac{n-1}{2}}\perp \langle \varepsilon\rangle, \quad   \mathbf{H}^{\frac{n-3}{2}}\perp 2^{-1}\pi A(2,2\rho)\perp \langle\Delta\varepsilon\rangle,  \\
			 &\mathbf{H}^{\frac{n-1}{2}}\perp \langle\varepsilon\pi\rangle,   \quad  \mathbf{H}^{\frac{n-3}{2}}\perp 2^{-1}A(2,2\rho)\perp \langle\Delta\varepsilon\pi\rangle
	\end{split}	\]for all $\varepsilon\in \mathcal{U}$.
		\end{enumerate}
	\end{thm}

Recall from \cite[82:18a and 91:2]{omeara_quadratic_1963} that for every quadratic space $ V $, up to isometry, there is exactly one $ \mathcal{O}_{F} $-maximal lattice on $ V $. Thus, the total number of lattices listed in Theorem\;\ref{thm:nuniversaldyadic15theorem} is equal to the number of classes of all $ n $-dimensional quadratic spaces. This number is $2^{[F:\mathbb{Q}_{2}]+3}$ if $n\ge 3$, or $2^{[F:\mathbb{Q}_{2}]+3}-1$ if $n=2$.

Indeed, \cite[63:20]{omeara_quadratic_1963} shows that a quadratic space  of dimension $n$ is determined uniquely by its determinant and Hasse symbol. When $ n\ge 3 $, there are exactly two quadratic spaces with the same determinant but with opposite Hasse symbols (\cite[63:22]{omeara_quadratic_1963}). The same holds for $ n=2 $ except in the case of determinant $-1$, the only binary space with determinant $ -1 $ being the hyperbolic plane $ \mathbb{H} $. So, by \cite[63:9 and 16:4]{omeara_quadratic_1963},  the total number of isometry classes of $n$-dimensional quadratic spaces over $F$ is $ 2|F^{\times}/F^{\times 2}|=8(N\mathfrak{p})^{e}=2^{[F:\mathbb{Q}_{2}]+3} $ if $n\ge 3$, or $2^{[F:\mathbb{Q}_{2}]+3}-1$ if $n=2$.

 \medskip

	The rest of the paper is organized as follows. In Section \ref{sec:pre}, we recall some notations and results from  Beli's papers that will be used in later proofs.
	In Section \ref{sec:construction} we determine all the  $\mathcal{O}_F$-maximal lattices of rank $ n $ and prove  preliminary properties of them. We prove that these lattices are precisely the lattices listed in Theorem \ref{thm:nuniversaldyadic15theorem} and thus obtain a proof of that theorem. Necessary and sufficient conditions for $n$-universality will be established in Sections  \ref{sec:even-con}  and  \ref{sec:odd-con} for even and odd $n$ respectively. For even $n\ge 2$, a more concise criterion for $n$-university will be shown in Theorem \ref{thm:2universaldyadicver2}. 	Section \ref{sec:proof-main} is devoted to a proof of Theorem \ref{thm:nuniversaldyadic}.

	\section{Representation theory using BONGs}\label{sec:pre}
	
	We briefly review some definitions and preliminary results from Beli's representation theory established in the series papers \cite{beli_thesis_2001,beli_integral_2003,beli_representations_2006,beli_Anew_2010,beli_representations_2019,beli_universal_2020}. The reader is referred to these papers for any
	unexplained notation and definition.
	
	\begin{defn}\label{defn:BONG}
		Let $ M$ be an $ \mathcal{O}_{F} $-lattice. A vector $ x\in M$ is called a \textit{norm generator} of $ M $ if $ \mathfrak{n}M=Q(x)\mathcal{O}_{F} $. A sequence of vectors $ x_{1},\ldots,x_{m}$ in $FM$ is called a \textit{Basis Of Norm Generators} (BONG) for $M$ if $ x_{1} $ is a norm generator for $ M $ and $ x_{2},\ldots,x_{m} $ is a BONG for $ \pr_{x_{1}^{\perp}}M$, where $ \pr_{x_{1}^{\perp}} $ denotes the projection from $ FM$ to $(Fx_{1})^{\perp} $, the orthogonal complement of $Fx_1$ in $FM$.
	 \end{defn}
	
		A BONG $ x_{1},\ldots,x_{m} $ is said to be \textit{good} if $ \ord\, (Q(x_{i}))\le \ord\,(Q(x_{i+2}))$ for all $ 1\le i\le m-2 $. If $ x_{1},\ldots,x_{m} $ is a good BONG for $M $,
		we define $ R_{i}(M):=\ord(Q(x_i))$. We have $ \mathfrak{n}M=Q(x_{1})\mathcal{O}_{F}=\mathfrak{p}^{R_{1}} $, so $M$ is integral if and only if
		\begin{align}\label{eq:integralcondition}
			R_{1}\ge 0\,.
		\end{align}
	
	Every lattice possesses a good BONG (see \cite[Lemma 4.6]{beli_integral_2003} for a proof and \cite[\S 7]{beli_representations_2006} for an algorithm) and the invariants $R_{i}(M)$ are independent of the choice of the good BONG (\cite[Lemma 4.7]{beli_integral_2003}).
		
		By \cite[Corollary 2.6]{beli_integral_2003}, a BONG $ x_{1},\ldots,x_{m} $ uniquely determines a lattice $ L $. Consequently, the class of $ L $ is uniquely determined by $ a_{1},\ldots, a_{m} $, where $ a_{i}=Q(x_{i}) $. Therefore, we will say that $ L\cong \prec a_{1},\ldots,a_{m}\succ $ relative to the BONG $ x_{1},\ldots,x_{m} $.

 For $ a_{1},\ldots, a_{m}\in F^{\times } $,  the expression $ \prec a_{1},\ldots,a_{m}\succ $ is well defined only if there is a lattice $ L$ with $ L\cong \prec a_{1},\ldots,a_{m}\succ $ relative to some BONG.

	\begin{lem}\label{lem:2.2}
		Let $ x_{1},\ldots, x_{m} $ be pairwise orthogonal vectors of a quadratic space with $ Q(x_{i})=a_{i} $ and $ R_{i}=\ord (a_{i})$.
		
		Then $ x_{1},\ldots,x_{m} $ is a good BONG for some lattice if and only if
		\begin{equation}\label{eq:GoodBONGs}
			R_{i}\le R_{i+2} \quad \text{for all }\;  1\le i\le m-2
		\end{equation}
		and
		\begin{equation}\label{eq:BONGs}
			R_{i+1}-R_{i}+d(-a_{i}a_{i+1})\ge 0 \quad\text{and}\quad  R_{i+1}-R_{i}\ge -2e \quad \text{for all }\; 1\le i\le m-1\,.
		\end{equation}
	\end{lem}
	\begin{proof}
		See \cite[Lemmas 3.5, 3.6 and 4.3(ii)]{beli_integral_2003}.
	\end{proof}

	In the remainder of this section, let $M\cong \prec a_{1},\ldots, a_{m} \succ $ be an $ \mathcal{O}_{F} $-lattice  relative to some good BONG and $ R_{i}=R_{i}(M)$.

	\begin{cor}\label{cor:R-R-odd}
	   (i) If $ R_{i+1}-R_{i} $ is odd, then  $ R_{i+1}-R_{i}>0 $. Equivalently, if $ R_{i+1}-R_{i}\le 0 $, then $ R_{i+1}-R_{i} $ must be even.
	
	   (ii) If $ R_{i+1}-R_{i}=-2e $, then $ d(-a_{i}a_{i+1})\ge 2e $ and $ \prec a_{i},a_{i+1}\succ \cong 2^{-1}\pi^{R_{i}}A(0,0)$ or $ 2^{-1}\pi^{R_{i}}A(2,2\rho) $. Consequently, $[a_{i},a_{i+1}]\cong \mathbb{H}$ or $[\pi^{R_{i}},-\Delta\pi^{R_{i}}]$.
	\end{cor}
	\begin{proof}
		 (i) If $ R_{i+1}-R_{i} $ is odd, then $\ord(a_ia_{i+1})=R_i+R_{i+1}$ is odd, so $ d(-a_{i}a_{i+1})=0 $ and hence $ R_{i+1}-R_{i}\ge 0 $ by \eqref{eq:BONGs}. Therefore,  $ R_{i+1}-R_{i}\ge 1>0 $.
		
		(ii) Suppose $ R_{i+1}-R_{i}=-2e $. Then $ d(-a_{i}a_{i+1})\ge R_{i}-R_{i+1}=2e $ by \eqref{eq:BONGs}, and $ (R_{i+1}+R_{i})/2=(R_{i}-2e+R_{i})/2=R_{i}-e $. Thus the lattice $L:=\prec a_i,\,a_{i+1}\succ$ satisfies  $ \mathfrak{n}(L)=2\mathfrak{s}(L)=\mathfrak{p}^{R_{i}} $ by \cite[Corollary 3.4(iii)]{beli_integral_2003}. So $L\cong 2^{-1}\pi^{R_{i}}A(0,0)$ or $ 2^{-1}\pi^{R_{i}}A(2,2\rho)   $ by \cite[93:11]{omeara_quadratic_1963}.
	\end{proof}

	\begin{defn}\label{defn:alpha}
		For $ 1\le i\le m-1 $, we put
		\begin{align}\label{T}
			T_{0}^{(i)}=\dfrac{R_{i+1}-R_{i}}{2}+e, \quad T_{j}^{(i)}=
			\begin{cases}
				R_{i+1}-R_{j}+d(-a_{j}a_{j+1}) &\text{if $ 1\le j\le i $}\,, \\
			 R_{j+1}-R_{i}+d(-a_{j}a_{j+1})  &\text{if $ i\le j\le m-1 $}\,,
			\end{cases}
		\end{align}
		and define	$ \alpha_{i}=\alpha_{i}(M) $ to be the minimum of the set $ \{T_{0}^{(i)},\ldots, T_{m-1}^{(i)}\} $.

Let $  c_{1},c_{2},\ldots \in F^{\times} $. For $ 1\le i\le j+1 $, we write $ c_{i,j}=c_{i}\cdots c_{j} $ for short and set $ c_{i,i-1}=1 $. For $ 0\le i-1\le j\le m $, we define
	\begin{equation}\label{defn:d[]}
  d[ca_{i,j}]:=\min\{d(ca_{i,j}),\alpha_{i-1},\alpha_{j}\}\,,\quad c\in F^{\times}\,.
\end{equation}	Here, if $ i-1=0$ or $ m $, $ \alpha_{i-1} $ is ignored; if $ j=0 $ or $ m $, $ \alpha_{j} $ is ignored.  By \cite[Corollary\;2.5(i)]{beli_Anew_2010}, we have the following frequently used formula:
	\begin{align}\label{eq:alpha-defn}
		\alpha_{i}=\min\{(R_{i+1}-R_{i})/2+e,R_{i+1}-R_{i}+d[-a_{i,i+1}]\}\,.
	\end{align}
	The above quantities $ \alpha_{i}(M) $ and $ d[ca_{i,j}] $ are independent of the choice of good BONG (see \cite[\S 2]{beli_Anew_2010} and \cite[\S 4]{beli_representations_2006}).
\end{defn}

	In the next two propositions we collect some useful properties of the invariants  $ R_{i} $ and $ \alpha_{i} $.
	
	\begin{prop}\label{prop:Ralphaproperty1}
		Suppose  $ 1\le i\le j\le m-1 $.
		
		\begin{enumerate}
			\item[(i)]  We have $R_i+\alpha_i\le R_j+\alpha_j$ and $ -R_{i+1}+\alpha_{i}\ge -R_{j+1}+\alpha_{j} $.
			\item[(ii)] If  $ R_{i}+R_{i+1}=R_{j}+R_{j+1} $, then $ R_{i}+\alpha_{i}=\cdots=R_{j}+\alpha_{j} $.
		\end{enumerate}
	\end{prop}
	\begin{proof}
		See \cite[Lemma\;2.2 and Corollary\;2.3(i)]{beli_Anew_2010}.
	\end{proof}
	
	\begin{prop}\label{prop:Ralphaproperty2}
		Suppose  $ 1\le i\le m-1 $.
		
		\begin{enumerate}
			\item[(i)] Either $\alpha_i\in [0,\,2e]\cap \Z$ or $\alpha_i\in (2e,\,\infty)\cap \frac{1}{2}\Z$. In particular, $\alpha_i\ge 0$.
			
			Moreover, $\alpha_i=0$ if and only if $R_{i+1}-R_i=-2e$.
			\item[(ii)] $ R_{i+1}-R_{i}>2e$ (resp. $ =2e $, $ <2e $) if and only if $ \alpha_{i}>2e $ (resp. $ =2e $, $ <2e $).
			\item[(iii)] Suppose $ R_{i+1}-R_{i}\le 2e $. Then $ \alpha_{i}\ge R_{i+1}-R_{i} $, and equality holds if and only if $ R_{i+1}-R_{i}=2e $ or $ R_{i+1}-R_{i} $ is odd.	
			\item[(iv)] If $ R_{i+1}-R_{i}\ge 2e$ or $ R_{i+1}-R_{i}\in \{-2e,2-2e,2e-2\} $, then $ \alpha_{i}=(R_{i+1}-R_{i})/2+e $.
		
			\item[(v)] If $ \alpha_{i}=0 $, or equivalently $R_{i+1}-R_i=-2e$ by (i), then $  d[-a_{i}a_{i+1}]\ge 2e $.
			
			\item[(vi)] Suppose $ \alpha_{i}=1 $. Then $ R_{i+1}-R_{i}\in [2-2e,\,0]^{E}\cup\{1\} $. Moreover, we have $ d[-a_{i,i+1}]\ge R_{i}-R_{i+1}+1 $ and equality holds
			if $R_{i+1}-R_{i}\not=2-2e $.
			
			\item[(vii)] If $ 2-2e<R_{i+1}-R_{i}\le 0 $, then $ \alpha_{i}=1 $ if and only if $ d[-a_{i,i+1}]=R_{i}-R_{i+1}+1 $.
			
			(Notice also that when $R_{i+1}-R_i\in \{2-2e,\,1\}$, we have $\alpha_i=1$ by (iii) and (iv).)
		\end{enumerate}
	\end{prop}
	\begin{proof}
		See \cite[Lemma 2.7, Corollaries\;2.8, 2.9]{beli_Anew_2010} for (i)--(v) and \cite[Lemma 2.8]{beli_universal_2020} for (vi) and (vii).
	\end{proof}
	
	\begin{prop}\label{prop:Ralphaproperty3}
		Suppose that $ M $ is integral.

		(i) We have $ R_{j}\ge R_{i}\ge 0 $ for all  $ i,j\in [1,\,m]^O$ with $ j\ge i $ and $ R_{j}\ge R_{i}\ge -2e $ for all $ i,j \in [1,\,m]^E$ with $ j\ge i $.
		
		(ii) If $ R_{j}=0 $ for some $ j\in [1,m]^{O} $, then $ R_{i}=0 $ for all  $ i\in [1,j]^{O} $ and $ R_{i} $ is even for all $ 1\le i\le j $.
		
		(iii) If $ R_{j}=-2e$ for some $ j\in [1,m]^{E} $, then for each $ i\in [1,j]^{E} $, we have $ R_{i-1}=0 $, $ R_{i}=-2e $ and $ d(-a_{i-1}a_{i})\ge d[-a_{i-1}a_{i}]\ge 2e $. Consequently, $ d[(-1)^{j/2}a_{1,j}]\ge 2e $.
		
	    (iv)  If $ R_{j}=-2e$ for some $ j\in [1,m]^{E} $, then $ [a_{1},\ldots,a_{j}]\cong \mathbb{H}^{j/2}$ or $\mathbb{H}^{(j-2)/2}\bot[1,\,-\Delta]$.
		
		(v) If $ R_{j}=-2e$ and $ R_{j+1} $ is even for some $ j\in [1,m]^{E} $, then $ [a_{1},\ldots,a_{j+1}]\cong \mathbb{H}^{j/2}\bot [\varepsilon]$ for some $ \varepsilon\in \mathcal{O}_{F}^{\times} $ with $ \varepsilon \in a_{j+1}F^{\times 2}\cup  \Delta a_{j+1} F^{\times 2} $.
	\end{prop}
	\begin{proof}
		(i) For odd indices $ j\ge i $, we have $ R_{j}\ge R_{i}\ge R_{1}\ge 0 $ by \eqref{eq:GoodBONGs} and \eqref{eq:integralcondition}. For even indices $ j\ge i $, we have $ R_{j}\ge R_{i} \ge R_{i-1}-2e\ge -2e $ by \eqref{eq:GoodBONGs} and \eqref{eq:BONGs}.
		
		(ii) For $ i\in [1,j]^{O} $, $ 0\le  R_{i}\le R_{j}=0 $ by (i) and hence $ R_{i}=0 $. Suppose that there exists $ i_{0}\in [1,j-1]^{E} $ for which $ R_{i_{0}} $ is odd. Then $ (R_{i_{0}+1}-R_{i_{0}})(R_{i_{0}}-R_{i_{0}-1})=-R_{i_{0}}^{2}\le 0 $. But both $ R_{i_{0}+1}-R_{i_{0}}$ and  $ R_{i_{0}}-R_{i_{0}-1} $ are positive by Corollary \ref{cor:R-R-odd}(i), so we get a contradiction.
		
		(iii) For $ i\in [1,j]^{E} $, $ -2e\le R_{i}\le R_{j}=-2e $ by (i) and hence $ R_{i}=-2e $. Since $ -2e-R_{i-1}=R_{i}-R_{i-1}\ge -2e $ by \eqref{eq:BONGs}, $ R_{i-1}\le 0 $ and so $ R_{i-1}=0 $ by (i). It follows that $ d(-a_{i-1}a_{i})\ge d[-a_{i-1}a_{i}]\ge 2e $ by Proposition \ref{prop:Ralphaproperty2}(v). Hence $ d[(-1)^{j/2}a_{1,j}]\ge 2e $ by the domination principle.

		(iv) For $ i\in [1,j]^{E} $, we have $R_{i-1}=0 $ and $ R_{i}=-2e $ by (iii) and so $ [a_{i-1},a_{i}]\cong \mathbb{H} $ or $ [1,-\Delta] $ by Corollary \ref{cor:R-R-odd}(ii). Note that $ \mathbb{H}^{2}\cong [1,-\Delta]\perp [1,-\Delta] $ and hence
		$ [a_{1},\ldots,a_{j}]$ is isometric to either $\mathbb{H}^{j/2}$ or $ \mathbb{H}^{(j-2)/2}\perp [1,-\Delta]$.
		
		(v) Since $R_{j+1}=\ord(a_{j+1})$ is even, we can find a unit $\eta$ in $a_{j+1}F^{\times 2}$. Since $(\Delta,\,-\eta)_{\p}=1$ (\cite[63:11a]{omeara_quadratic_1963}), the ternary space $[1,\,-\Delta,\,\eta]$ is isotropic and hence isometric to $\mathbb{H}\bot [\Delta\eta]$. By (iv), $ [a_{1},\ldots,a_{j+1}]$ is isometric to either $ \mathbb{H}^{j/2}\perp [a_{j+1}] \cong \mathbb{H}^{j/2}\perp [\eta] $ or $ \mathbb{H}^{(j-2)/2}\perp [1,-\Delta,a_{j+1}]\cong  \mathbb{H}^{(j-2)/2}\perp [1,-\Delta,\eta] \cong\mathbb{H}^{j/2}\perp [\Delta\eta]$. Choosing $\varepsilon\in\{\eta,\,\Delta\eta\}$ accordingly completes the proof.
	\end{proof}

	Now consider two  $ \mathcal{O}_{F} $-lattices $ M\cong \prec a_{1},\ldots,a_{m} \succ$ and $ N\cong \prec b_{1},\ldots, b_{n}\succ $ relative to some good BONGs and suppose $ m\ge n $. Let $ R_{i}=R_{i}(M) $, $ S_{i}=R_{i}(N) $, $ \alpha_{i}=\alpha_{i}(M) $ and $ \beta_{i}=\alpha_{i}(N) $.
	For $ 0\le i,j\le m $, we define
	\begin{equation}\label{defn:d[ab]}
		d[ca_{1,i}b_{1,j}]=\min\{d(ca_{1,i}b_{1,j}),\alpha_{i},\beta_{j}\}\,,\quad c\in F^{\times}\,.
	\end{equation}
	Here if $ i=0 $ or $ m $, then $ \alpha_{i} $ is ignored; if $ j=0 $ or $ j=n $, $ \beta_{j} $ is ignored. Note that the quantity
 $d[ca_{i,j}]$ defined in \eqref{defn:d[]} coincides with $d[ca_{1,i-1}a_{1,j}]$.

 We have a domination principle for  $ d[ca_{1,i}b_{1,j}] $  (cf. \cite[\S 1.1, p.\hskip 0.1cm 6]{beli_representations_2019}). Namely, for another lattice $ L\cong \prec c_{1},\ldots,c_{k}\succ $ relative to a good BONG, we have
	\begin{align*}
		d[cc^{\prime}a_{1,i}c_{1,\ell}]\ge \min\{d[ca_{1,i}b_{1,j}],d[c^{\prime}b_{1,j}c_{1,\ell}]\}\,,\quad c,c^{\prime}\in F^{\times}\,.
	\end{align*}

	  For any $ 1\le i\le \min\{m-1,n\}$, we define
	\[
	\begin{split}
		A_{i}=A_{i}(M,N):=\min\{&(R_{i+1}-S_{i})/2+e,R_{i+1}-S_{i}+d[-a_{1,i+1}b_{1,i-1}],\\
		&R_{i+1}+R_{i+2}-S_{i-1}-S_{i}+d[a_{1,i+2}b_{1,i-2}]\}\,,
	\end{split}
	\]where  the term $ R_{i+1}+R_{i+2}-S_{i-1}-S_{i}+d[a_{1,i+2}b_{1,i-2}] $ is ignored if $ i=1 $ or $ m-1 $. It can be shown that  $ d[ca_{1,i}b_{1,j}] $ and $ A_{i}(M,N) $ are independent of the choice of good BONG (\cite[\S 4]{beli_representations_2006}).
	
	\medskip

	Taking the remarks following \cite[Theorem 2.1]{beli_representations_2019} into account (cf. \cite[Lemma 2.16]{beli_representations_2019} for details), we can restate   \cite[Theorem 4.5]{beli_representations_2006} as follows:
	
	\begin{thm}\label{thm:beligeneral}
		Suppose $  n\le m$. Then $ N\rep M $ if and only if  $ FN\rep FM $ and the following conditions hold:
		
		(i) For any $ 1\le i\le n $, we have either $ R_{i}\le S_{i} $, or $ 1<i<m $ and $ R_{i}+R_{i+1}\le S_{i-1}+S_{i} $.
		
		(ii) For any $ 1\le i\le \min\{m-1,n\} $, we have $ d[a_{1,i}b_{1,i}]\ge A_{i} $.
		
		(iii) For any $ 1<i\le \min\{m-1,n+1\} $, if
		\begin{align}\label{eq:assumption(3)(ii)'}
			R_{i+1}>S_{i-1} \quad
			\text{and}\quad d[-a_{1,i}b_{1,i-2}]+d[-a_{1,i+1}b_{1,i-1}]>2e+S_{i-1}-R_{i+1}\;,
		\end{align}
		then $ [b_{1},\ldots, b_{i-1}]\rep [a_{1},\ldots,a_{i}] $.

		(iv) For any $ 1<i\le \min\{m-2,n+1\} $ such that $ S_{i}\ge R_{i+2}>S_{i-1}+2e\ge R_{i+1}+2e$, we have $ [b_{1},\ldots,b_{i-1}]\rep [a_{1},\ldots,a_{i+1}] $. (If $ i=n+1 $, the condition $ S_{i}\ge R_{i+2} $ is ignored.)
	\end{thm}

	\begin{lem}\label{lem:Aj}
		Suppose that $ M $ and $ N $ are integral. Let $  j\in[ 1,\min\{m-1,n\}]^{E} $. If $ R_{j}=-2e $ and $ R_{j+1}=0 $, then $ d[a_{1,j}b_{1,j}]\ge A_{j} $, i.e. Theorem \ref{thm:beligeneral}(ii) holds at the index $ j $.
	\end{lem}
	\begin{proof}
		 From Proposition \ref{prop:Ralphaproperty3}(i), we see $ S_{j}\ge -2e $. Since $ R_{j+1}=0 $ and $ R_{j}=-2e $, we have
		\begin{align*}
			A_{j}&\le \dfrac{R_{j+1}-S_{j}}{2}+e\le\dfrac{0-(-2e)}{2}+e= 2e\le d[(-1)^{j/2}a_{1,j}]\,,
		\end{align*}
	where the last inequality follows by Proposition \ref{prop:Ralphaproperty3}(iii). The hypothesis $ R_{j+1}=0 $ also implies	
\[
			A_{j}\le R_{j+1}-S_{j}+d[-a_{1,j+1}b_{1,j-1}]\le R_{j+1}-S_{j}+\beta_{j-1}=-S_{j}+\beta_{j-1}\,.
\]
	By Proposition \ref{prop:Ralphaproperty3}(i), we have $ S_{i-1}\ge 0 $ for every even $ i $. It follows that
	\begin{align}\label{Aj-3}
		 A_{j}\le -S_{j}+\beta_{j-1}\le -S_{i}+\beta_{i-1}  \le S_{i-1}-S_{i}+\beta_{i-1}\le d[-b_{i-1}b_{i}]
\quad\text{ for } i\in [1,j]^{E} \,,
	\end{align}
 where the second inequality holds by Proposition \ref{prop:Ralphaproperty1}(i) and the last inequality follows from \eqref{eq:alpha-defn}. By the domination principle, \eqref{Aj-3} implies $ A_{j}\le d[(-1)^{j/2}b_{1,j}] $. Since also $ A_{j}\le d[(-1)^{j/2}a_{1,j}]$, we deduce,  by the domination principle, that $ A_{j}\le d[a_{1,j}b_{1,j}] $.
	\end{proof}

		\begin{lem}\label{lem:d[-ai+1ai+2]=1-Ri+2}
		Suppose that $ M$ and $N $ are integral. Let $  i\in [1,\min\{m-2,n\}]^{E} $. Suppose $ R_{i}=-2e $, $ R_{i+1}=0 $ and $ R_{i+2}>-2e $.
		
		(i) We have $ d[-a_{i+1}a_{i+2}]\ge 1-R_{i+2}$. Also, if $ d[-a_{i+1}a_{i+2}]=1-R_{i+2}$, then $ \alpha_{i+1}=1 $ and $ R_{i+2}\in [2-2e,0]^{E}\cup \{1\} $.
		
		(ii) $ d[-a_{i+1}a_{i+2}]=1-R_{i+2} $ if and only if $ d[(-1)^{(i+2)/2}a_{1,i+2}]=1-R_{i+2} $.
		
		(iii) If $ d[-a_{i+1}a_{i+2}]=1-R_{i+2} $, then one of the following statements holds:
		
		(1) $ d[-a_{1,i+2}b_{1,i}]=1-R_{i+2} $.
		
		(2) There exists some $ j\in [1,i]^{E} $ such that $ R_{i+2}\le S_{j} $ and $ \beta_{k}\le S_{k+1}-S_{j-1}+1-R_{i+2}\le S_{k+1}+1-R_{i+2} $ for each $ j-1\le k\le n-1 $.
	\end{lem}
	\begin{proof}		
		(i) Since $ R_{i+2}-R_{i+1}=R_{i+2}>-2e $, we have $ \alpha_{i+1}\ge 1 $ by Proposition \ref{prop:Ralphaproperty2}(i). Hence $ R_{i+2}+d[-a_{i+1}a_{i+2}]=R_{i+2}-R_{i+1}+d[-a_{i+1}a_{i+2}]\ge \alpha_{i+1}\ge 1 $ by \eqref{eq:alpha-defn}, and $ \alpha_{i+1}=1 $ if moreover $ d[-a_{i+1}a_{i+2}]=1-R_{i+2} $. Further, $\alpha_{i+1}=1$ implies $ R_{i+2}\in [2-2e,0]^{E}\cup \{1\} $ by Proposition \ref{prop:Ralphaproperty2}(vi).
		
		(ii) By Proposition \ref{prop:Ralphaproperty3}(iii), we have $ d[(-1)^{i/2}a_{1,i}]\ge 2e>1-R_{i+2} $ and so the statement holds by the domination principle.
		
		(iii) If $ d[-a_{1,i+2}b_{1,i}]\not=1-R_{i+2}=d[(-1)^{(i+2)/2}a_{1,i+2}] $,  then by the domination principle, $
		d[(-1)^{i/2}b_{1,i}]=\min\{d[-a_{1,i+2}b_{1,i}],d[(-1)^{(i+2)/2}a_{1,i+2}]\} \le 1-R_{i+2} $. Then, again by the domination principle,
		\begin{align}\label{d-bj-1bj}
			d[-b_{j-1}b_{j}]=\min\limits_{k\in [1,i]^{E}}\{d[-b_{k-1}b_{k}]\}\le d[(-1)^{i/2}b_{1,i}]\le 1-R_{i+2}
		\end{align}
		for some $ j\in [1,i]^{E} $. Note that $ S_{j-1}\ge S_{1}\ge 0 $. Hence for each $ j-1\le k\le n-1 $,
		\[
\begin{split}
			-S_{k+1}+\beta_{k}&\le -S_{j}+\beta_{j-1}\;\;\text{(by Proposition \ref{prop:Ralphaproperty1}(i))}\\
			&\le -S_{j-1}+d[-b_{j-1}b_{j}]\;\;\text{(by \eqref{eq:alpha-defn})}\\
&\le -S_{j-1}+(1-R_{i+2})\;\;\text{(by \eqref{d-bj-1bj})}\\
&\le 1-R_{i+2}\,,		
\end{split}\]
		i.e. $ \beta_{k}\le S_{k+1}-S_{j-1}+1-R_{i+2} \le S_{k+1}+1-R_{i+2} $.

Since $ d[-b_{j-1}b_{j}]\le 1-R_{i+2}<2e $, we have $ S_{j}-S_{j-1}\not=-2e $ by Proposition \ref{prop:Ralphaproperty2}(v) and so $ \beta_{j-1}\ge 1 $ by Proposition \ref{prop:Ralphaproperty2}(i). Hence by \eqref{eq:alpha-defn},
		\begin{align*}
			1-R_{i+2}\ge d[-b_{j-1}b_{j}]\ge  S_{j-1}-S_{j}+\beta_{j-1}\ge 0-S_{j}+1\,,
		\end{align*}
		i.e. $ R_{i+2}\le S_{j} $.
	\end{proof}
	
		\begin{lem}\label{lem:da1nb1n}
	 Suppose that $ M$ and $N $ are integral, $n\ge 3 $ is odd, $ R_{n-1}=-2e $ and $ R_{n}=0 $. If $ \alpha_{n}=1 $ and $d[-a_na_{n+1}]=1-R_{n+1}$, then $ R_{n+1}-S_{n}+d[-a_{1,n+1}b_{1,n-1}]\le d[a_{1,n}b_{1,n}] $.
	\end{lem}
	\begin{proof}
		Since $ \alpha_{n}=1 $, we have
		\begin{align*}
			d[a_{1,n}b_{1,n}]&=\min\{d(a_{1,n}b_{1,n}),1\}=
			\begin{cases}
				1 \quad &\text{if $ \ord(a_{1,n}b_{1,n}) $ is even}\,,  \\
				0  \quad &\text{if $ \ord(a_{1,n}b_{1,n})  $ is odd}\,.
			\end{cases}
		\end{align*}
		Note that $ R_{1},\ldots, R_{n}\in \{0,-2e\} $ by the hypothesis and Proposition \ref{prop:Ralphaproperty3}(iii). So $ \ord(a_{1,n}b_{1,n})=\sum_{k=1}^{n}R_{k}+\sum_{k=1}^{n}S_{k}\equiv \sum_{k=1}^{n}S_{k}\pmod{2} $. Hence
		\begin{align}\label{iffda1nb1n}
			d[a_{1,n}b_{1,n}]=0\quad\text{if and only if}\quad \text{$ \sum_{k=1}^{n}S_{k} $ is odd}\,.
		\end{align}
		
Suppose that $ d[-a_{1,n+1}b_{1,n-1}]=1-R_{n+1} $. We have $ S_{n}\ge 0 $ by Proposition \ref{prop:Ralphaproperty3}(i). If $S_n\ge 1$, then $1-S_{n}\le 0\le d[a_{1,n}b_{1,n}]$. If $ S_{n}=0 $, then $ S_{1},\ldots,S_{n} $ are even by Proposition \ref{prop:Ralphaproperty3}(ii). So from \eqref{iffda1nb1n} we obtain  $ 1-S_{n}=1=d[a_{1,n}b_{1,n}] $. Therefore, we have
		\begin{align*}
			R_{n+1}-S_{n}+d[-a_{1,n+1}b_{1,n-1}]=1-S_{n}\le d[a_{1,n}b_{1,n}]\,,
		\end{align*}as desired.

	Now suppose that $ d[-a_{1,n+1}b_{1,n-1}]\not=1-R_{n+1} $.  We claim that $ R_{n+1}-S_{n}+\beta_{n-1} \le  d[a_{1,n}b_{1,n}]$. Since $ d[-a_{n}a_{n+1}]=1-R_{n+1} $, by Lemma \ref{lem:d[-ai+1ai+2]=1-Ri+2}(iii) with $ i=k=n-1 $, there exists some $ j\in [1,n-1]^{E} $ such that
		\[
			R_{n+1}\le S_{j}\quad	\text{and}\quad 	\beta_{n-1}\le S_{n}-S_{j-1}+1-R_{n+1}\label{betan-1}\,.
		\]Note that the second inequality is equivalent to $R_{n+1}-S_n+\beta_{n-1}\le 1-S_{j-1}$.
		Since $ j-1 $ is odd, we have $ S_{j-1}\ge 0 $. So
		\begin{align*}
			R_{n+1}-S_{n}+\beta_{n-1}\le 1-S_{j-1}\le d[a_{1,n}b_{1,n}]\,
		\end{align*}
		except possibly when $ S_{j-1}=0$ and $d[a_{1,n}b_{1,n}]=0 $. Now consider this exceptional case. Then $ \sum_{k=1}^{n}S_{k} $ is odd by \eqref{iffda1nb1n}. Since $ S_{j-1}=0 $, $ S_{1},\ldots,S_{j-1} $ are even by Proposition \ref{prop:Ralphaproperty3}(ii). Hence $ \sum_{k=j}^{n}S_{k}\equiv \sum_{k=1}^{n}S_{k}\equiv 1\pmod{2} $. Since $j$ is even and $n$ is odd, the sum $ \sum_{k=j}^{n}S_{k} $ can be written as $ (S_{j}+S_{j+1})+\cdots+(S_{n-1}+S_{n}) $. Hence there exists some $ \ell \in [j,n-1]^{E} $ such that $ S_{\ell}+S_{\ell+1} $ is odd, i.e. $ d(-b_{\ell}b_{\ell+1})=0 $. Since the indices $ \ell,j $ are even and $ \ell\ge j $, we have $ S_{\ell}\ge S_{j}\ge R_{n+1} $  and hence
		\[
			\beta_{n-1}\le  S_{n}-S_{\ell}+d(-b_{\ell}b_{\ell+1})=S_{n}-S_{\ell}\le S_{n}-R_{n+1}\,,
		\]		by the definition of $\beta_{n-1}$ (cf. Definition \ref{defn:alpha}). So $ R_{n+1}-S_{n}+\beta_{n-1}\le 0\le d[a_{1,n}b_{1,n}] $ and thus the claim also holds in the exceptional case.

From the claim and the obvious inequality $ d[-a_{1,n+1}b_{1,n-1}]\le \beta_{n-1} $ we conclude that
$R_{n+1}-S_{n}+d[-a_{1,n+1}b_{1,n-1}]\le R_{n+1}-S_{n}+\beta_{n-1} \le d[a_{1,n}b_{1,n}]$, as required.
	\end{proof}

	\section{Maximal lattices and their BONGs}\label{sec:construction}

Obviously, a unary $\mathcal{O}_F$-lattice is $ \mathcal{O}_{F} $-maximal if and only if it is of the form  $ \prec\delta\succ $ or $ \prec \delta\pi\succ $, with $\delta\in\mathcal{O}_F^\times$.	Our goal in this section is to determine all the $n$-ary $\mathcal{O}_F$-maximal lattices for any $n\ge 2$. We also show that these lattices form a minimal testing set for the $n$-universal property.

	\medskip
	
	Recall that for any $c\in F^\times$, we have $d(c)\ge 2e$ if and only if $c\in F^{\times 2}\cup\Delta F^{\times 2}$, and $d(c)\ge 1$ if and only if $\ord(c)$ is even.
	
	\begin{defn}\label{defn:3.1}
		Let $ c\in F^{\times}\backslash (F^{\times 2}\cup\Delta F^{\times 2})$ and write $\delta=c\pi^{-\ord(c)}\in \mathcal{O}_{F}^{\times}$.  When $\ord(c)$ is even (or equivalently, $ 1\le d(c)=d(\delta) <2e $), we choose an expression
		\begin{align}\label{eq:expressionmu}
			\delta=s^{2}(1+r\pi^{d(\delta)})=s^{2}(1+r\pi^{d(c)}),\quad\text{with }\; r,s\in \mathcal{O}_{F}^{\times}\,
		\end{align}and put
		\begin{align*}
			c^{\#}:=
			\begin{cases}
				\Delta    &\text{if }\; \ord(c) \text{ is odd}\,, \\
				1+4\rho r^{-1}\pi^{-d(c)}  &\text{if }\; \ord(c) \text{ is even}\,.
			\end{cases}
		\end{align*}
	When $\ord(c)$ is even, the element $ c^{\#} $ depends on the expression \eqref{eq:expressionmu}.
	\end{defn}

	\begin{prop}\label{prop:duality2}
		For any $ c\in F^{\times}\backslash (F^{\times 2}\cup\Delta F^{\times 2})$, we have
		\[
		c^{\#} \in \mathcal{O}_F^\times\,,\; d(c^{\#})=2e-d(c) \quad\text{ and }\quad  (c^{\#},c)_{\mathfrak{p}}=-1\,.
		\]
	\end{prop}
	\begin{proof}
		It is clear from the definition of $ c^{\#} $, \cite[63:11a]{omeara_quadratic_1963} and \cite[Lemma 4.1]{hhx_indefinite_2021}.
	\end{proof}
	
	Recall from \cite[63:20 and 63:22]{omeara_quadratic_1963} that in general, for every $ n\ge 3 $ and every $ D\in F^{\times}/F^{\times 2} $, up to isometry, there are precisely two $ n $-ary quadratic spaces with determinant $ D $, one with Hasse symbol $1$ and the other with Hasse symbol $-1$. The same holds for $n=2$ and $D\not=-1$. The only exception is when $ n=2 $ and $ D=-1$, the only binary quadratic space of determinant $ -1 $ being  $ \mathbb{H} $.

	\begin{prop}\label{prop:dualspace}
		For any $ c\in F^{\times}\backslash (F^{\times 2}\cup\Delta F^{\times 2})$, up to isometry the quadratic space $[c^{\#},-c^{\#}c] $ is the only binary quadratic space with determinant $ -c $ that is not isometric to $ [1,-c] $.  In particular, its isometry class depends only on the square class of $c$ (and thus is independent of the expression $\eqref{eq:expressionmu}$).
	\end{prop}
	\begin{proof}
As we have mentioned above, there are only two isometry classes of binary spaces with determinant $-c$. So it is sufficient to check that the spaces $[c^{\#},-c^{\#}c] $ and $[1,\,-c]$ are not isometric. Indeed,  by Proposition \ref{prop:duality2}, we have $ (c^{\#},c)_{\mathfrak{p}}=-1 $, so $  [1,-c] $ does not represent $ c^{\#} $ and hence $ [c^{\#},-c^{\#}c]\not\cong [1,-c]$.
	\end{proof}

	\begin{defn}\label{defn:space}
		Let $ n\ge 2 $ be an integer and $ c\in F^{\times}/F^{\times 2} $. Define the $ n $-dimensional quadratic spaces $ W_{1}^{n}(c) $ and $ W_{2}^{n}(c) $ as follows:
		
		(i) If $ n $ is even, then $ W_{1}^{n}(c):=\mathbb{H}^{(n-2)/2}\perp [1,-c ]$ and $ W_{2}^{n}(c) $ is the quadratic space with $ \det W_{2}^{n}(c)=\det W_{1}^{n}(c)=(-1)^{n/2}c$ and $ W_{2}^{n}(c)\not\cong W_{1}^{n}(c) $.

Note that $ W_{2}^{n}(c) $ is defined in all cases except when $ n=2 $ and $ c=1 $. (The only binary space with determinant $-1$ is $W_1^2(1)=\mathbb{H}$.)
		
		(ii) If $ n $ is odd, then $ W_{1}^{n}(c):=\mathbb{H}^{(n-1)/2}\perp [c]$ and $ W_{2}^{n}(c) $ is the quadratic space with $ \det W_{2}^{n}(c)=\det W_{1}^{n}(c)=(-1)^{(n-1)/2}c $ and $ W_{2}^{n}(c)\not\cong W_{1}^{n}(c) $.
	\end{defn}

Let $\mathcal{U}$ be  a complete set of representatives of $\mathcal{O}_F^\times/\mathcal{O}_F^{\times 2}$. Then $ \{\delta,\delta\pi\mid \delta\in \mathcal{U}\} $ is a complete set of representatives for $ F^{\times }/F^{\times 2} $.

	\begin{prop}\label{prop:space}
	 	Let $n\ge 2$ be an integer and $ c\in F^{\times}/F^{\times2} $.
	 	
	 	(i) The quadratic spaces $ W_{1}^{n}(c) $ and $ W_{2}^{n}(c) $ are given by the following table:
	 	\begin{center}
	 		\renewcommand\arraystretch{1.5}
	 		\begin{tabular}{|c|c|c|c|}
	 			\hline
	 		 $ n $	& $ c $ & $ W_{1}^{n}(c) $  & $ W_{2}^{n}(c) $  \\
	 			\hline
	 			\multirow{4}*{\text{even}}	  & $ 1 $ & $ \mathbb{H}^{\frac{n}{2}} $ & $ \mathbb{H}^{\frac{n-4}{2}} \perp [1,-\Delta,\pi,-\Delta\pi]$ \text{($ n\ge 4 $)} \\
	 			\cline{2-4}
	 			& $ \Delta $  & $ \mathbb{H}^{\frac{n-2}{2}}\perp [1,-\Delta]  $ & $ \mathbb{H}^{\frac{n-2}{2}}\perp [\pi,-\Delta\pi] $  \\
	 			 \cline{2-4}
	 			& $\delta\in \mathcal{U}\backslash \{1,\Delta\} $  & $ \mathbb{H}^{\frac{n-2}{2}}\perp [1,-\delta]$ & $ \mathbb{H}^{\frac{n-2}{2}}\perp [\delta^{\#},-\delta^{\#}\delta]$  \\
	 			\cline{2-4}
	 			& $ \delta\pi$, with $\delta\in \mathcal{U} $  & $ \mathbb{H}^{\frac{n-2}{2}}\perp [1,-\delta\pi]$ & $ \mathbb{H}^{\frac{n-2}{2}}\perp [\Delta,-\Delta\delta\pi]$  \\
	 			\hline
	 			\multirow{2}*{\text{odd}} & $\delta\in\mathcal{U} $ & $ \mathbb{H}^{\frac{n-1}{2}}\perp [\delta] $  &   $ \mathbb{H}^{\frac{n-3}{2}}\perp [\pi,-\Delta\pi,\Delta\delta] $  \\
	 			\cline{2-4}
	 			& $ \delta\pi$, with $\delta\in\mathcal{U} $ & $ \mathbb{H}^{\frac{n-1}{2}}\perp [\delta\pi] $  &   $ \mathbb{H}^{\frac{n-3}{2}}\perp [1,-\Delta ,\Delta\delta\pi] $  \\
	 			\hline
	 		\end{tabular}
	 	\end{center}

 	    (ii) Every $ n $-dimensional quadratic space over $ F $ is isometric to one of the spaces in the table above.
 	
	 	(iii) For every $ n $-dimensional quadratic space $ W $, up to isometry, there is exactly one $ (n+2) $-dimensional quadratic space $ V $ that represents all $ n $-dimensional quadratic spaces except $ W $. Explicitly,  if $ W=W_{\nu}^{n}(c) $, with $ \nu\in \{1,2\} $, then $ V=W_{3-\nu}^{n+2}(c) $.
	\end{prop}
\begin{proof}
	(i) The description of $W^n_1(c)$ is clear from the definition.  The two quadratic spaces in each row of the table have the same determinant. So it suffices to show that they are not isometric. The quadratic space in the $ W_{2}^{n}(c) $ column writes as $ \mathbb{H}^{(n-k)/2}\perp U^{\prime} $, with $ \dim U^{\prime}=k $ for some $ k\le 4 $. We also have $ W_{1}^{n}(c)=\mathbb{H}^{(n-k)/2}\perp U $, where $ U=W_{1}^{k}(c) $. By Witt cancellation, we only have to prove that $ U\not\cong U^{\prime} $.
	
	Assume first that $ n $ is even. If $ c=1 $, then $ k=4 $, $ U=\mathbb{H}^{2} $ and $ U^{\prime}=[1,-\Delta,\pi,-\Delta\pi] $. We have $ U\not\cong U^{\prime} $ by \cite[63:17]{omeara_quadratic_1963}. If $ c\not=1 $, then $ k=2 $, $ U=[1,-c] $ and $ U^{\prime}=[c^{\prime},-c^{\prime}c] $ for some $ c^{\prime}\in F^{\times} $ with $ (c^{\prime},c)_{\mathfrak{p}}=-1 $. Since $ (c^{\prime},c)_{\mathfrak{p}}=-1 $, the binary space $  [1,-c] $ does not represent $ c^{\prime} $, so $ U=[1,-c]\not\cong U^{\prime}=[c^{\prime},-c^{\prime}c] $.
	
	If $ n $ is odd, then $ k=3 $, $ U=\mathbb{H}\perp [c] $ and $ U^{\prime}=[\pi,-\Delta\pi,\Delta\delta] $ or $ [1,-\Delta,\Delta\delta\pi] $. Note that $ (\Delta,-\Delta\delta\pi)_{\mathfrak{p}}=-1 $ and hence $ [1,-\Delta,\Delta\delta\pi] $ is anisotropic. Scaling by $ \pi $, we get that $ [\pi,-\Delta\pi,\Delta\delta] $ is also anisotropic. In both cases, $ U^{\prime} $ is anisotropic, so we cannot have $ U\cong U^{\prime} $.

	(ii) It is clear from the table because all possible determinants are exhausted and for every determinant we have two non-isometric quadratic spaces, except when $ n=2 $ and $ c=1 $. In the exceptional case, we only have $ W_{1}^{2}(1)=\mathbb{H} $.
	
	(iii) This follows from (ii) and \cite[63:21]{omeara_quadratic_1963}.
\end{proof}

\begin{defn}\label{defn:maximallattices}
		Let $ n\ge 2 $ be an integer, $ c\in F^{\times}/F^{\times 2} $ and $ \nu\in\{1,2\} $. We define $ N_{\nu}^{n}(c) $ as the $ \mathcal{O}_{F} $-maximal lattice on the space $ W_{\nu}^{n}(c) $ (cf. Definition \ref{defn:space}), provided that $ W_{\nu}^{n}(c) $ is defined. (Notice that $ N_{2}^{2}(1) $ is not defined.)
\end{defn}

	\begin{prop}\label{prop:maximallattices}
			Let $n\ge 2$ be an integer, $ c\in F^{\times}/F^{\times2} $ and $\mathcal{U}$ a complete system of representatives of $\mathcal{O}_F^\times/\mathcal{O}_F^{\times 2}$.

Then the $ \mathcal{O}_{F} $-maximal lattices $ N_{1}^{n}(c) $ and $ N_{2}^{n}(c) $ are given by the following table:
	\begin{center}
		\renewcommand\arraystretch{1.5}
		\begin{tabular}{|c|c|c|c|}
			\hline
			$ n $& $ c $ & $ N_{1}^{n}(c) $  & $ N_{2}^{n}(c) $  \\
			\hline
			\multirow{4}*{\text{even}}	 & $ 1 $ & $ \mathbf{H}^{\frac{n}{2}} $ & $ \mathbf{H}^{\frac{n-4}{2}} \perp 2^{-1}A(2,2\rho)\perp 2^{-1}\pi A(2,2\rho)$\text{($ n\ge 4 $)} \\
			\cline{2-4}
			& $ \Delta $  & $ \mathbf{H}^{\frac{n-2}{2}}\perp 2^{-1}A(2,2\rho)  $ & $ \mathbf{H}^{\frac{n-2}{2}}\perp 2^{-1}\pi A(2,2\rho) $  \\
			\cline{2-4}
			& $\delta\in \mathcal{U}\backslash \{1,\Delta\} $  & $ \mathbf{H}^{\frac{n-2}{2}}\perp \prec 1,-\delta\pi^{1-d(\delta)}\succ$ & $ \mathbf{H}^{\frac{n-2}{2}}\perp \prec\delta^{\#},-\delta^{\#}\delta\pi^{1-d(\delta)}\succ$  \\
			\cline{2-4}
			& $ \delta\pi$, with $\delta\in \mathcal{U} $  & $ \mathbf{H}^{\frac{n-2}{2}}\perp \langle  1,-\delta\pi\rangle$ & $ \mathbf{H}^{\frac{n-2}{2}}\perp \langle\Delta,-\Delta\delta\pi\rangle$  \\
			\hline
			\multirow{2}*{\text{odd}}	 & $\delta\in\mathcal{U} $ & $ \mathbf{H}^{\frac{n-1}{2}}\perp \langle \delta\rangle $  &   $ \mathbf{H}^{\frac{n-3}{2}}\perp 2^{-1}\pi A(2,2\rho)\perp \langle\Delta\delta\rangle $  \\
			 \cline{2-4}
			& $ \delta\pi$, with $\delta\in\mathcal{U} $ & $ \mathbf{H}^{\frac{n-1}{2}}\perp \langle\delta\pi\rangle $  &   $ \mathbf{H}^{\frac{n-3}{2}}\perp 2^{-1}A(2,2\rho)\perp \langle\Delta\delta\pi\rangle $  \\
			\hline
		\end{tabular}
	\end{center}
	\end{prop}

	\begin{re}\label{re:maximallattices}
		 For $ \delta\in \mathcal{U}\backslash \{1,\Delta\} $ with the property $ d(\delta)=\ord(\delta-1) $, we can deduce from  \cite[Corollary 3.4(iii)]{beli_integral_2003} and \cite[93:17]{omeara_quadratic_1963} that
		\begin{align*}
			\prec 1,-\delta\pi^{1-d(\delta)}\succ &\cong \pi^{\frac{1-d(\delta)}{2}}A(\pi^{\frac{d(\delta)-1}{2}},-(\delta-1)\pi^{\frac{1-d(\delta)}{2}})\,, \\
			\prec \delta^{\#},-\delta^{\#}\delta\pi^{1-d(\delta)}\succ &\cong \delta^{\#} \pi^{\frac{1-d(\delta)}{2}}A(\pi^{\frac{d(\delta)-1}{2}},-(\delta-1)\pi^{\frac{1-d(\delta)}{2}})\,,
		\end{align*}
		where $ \delta^{\#}=1+4\rho(\delta-1)^{-1} $.

	Note that for any $\varepsilon\in \mathcal{O}_F^\times$ one can find a $\delta\in \mathcal{O}_F^\times$ such that $\varepsilon\delta^{-1}\in \mathcal{O}_F^{\times 2}$ and  $d(\delta)=\ord(\delta-1)$. So, indeed, the set $ \mathcal{U} $ in Proposition \ref{prop:maximallattices} can be chosen  such that $d(\delta)=\ord(\delta-1)$ for all $\delta\in \mathcal{U}$.
	\end{re}
Let us first determine the $R_i$ invariants for the lattices listed in  Proposition \ref{prop:maximallattices}. We begin with the following observation.

\begin{lem}\label{lem3.10new}
(i) We have $\mathbf{H}=2^{-1}A(0,0)\cong  \prec 1,\,-\pi^{-2e}\succ$ and
\[
 2^{-1}A(2,\,2\rho)\cong \prec 1,\,-\Delta\pi^{-2e}\succ\,,\; 2^{-1}\pi A(2,\,2\rho)\cong \prec \pi,\,-\Delta\pi^{1-2e}\succ.
\]
(ii) For any $\delta\in \mathcal{U} $, we have
\[
		\langle 1,\,-\delta \pi\rangle\cong \prec 1,\,-\delta\pi\succ  \,,\;\; \langle \Delta,\,-\Delta\delta \pi\rangle\cong \prec \Delta,\,-\Delta\delta\pi\succ.
\]
(iii) Let $\kappa\in \cO_F^{\times}$ be such that $d(\kappa)=2e-1$ and let $\delta\in\cO_F^\times$. Then
\[
2^{-1}\pi A(2,2\rho) \perp \langle \Delta\delta\rangle\cong \prec \delta\kappa^{\#},-\delta\kappa^{\#}\kappa\pi^{2-2e}, \delta\kappa\succ\,.
\]
\end{lem}
\begin{proof}
(i) This follows by combining \cite[Lemma\;3.3(iii)]{beli_integral_2003} with \cite[93:11]{omeara_quadratic_1963}.

  (ii) This is a special case of \cite[Lemma 3.3(ii)]{beli_integral_2003}.

  (iii) Let $  a_{1}=\delta\kappa^{\#}$, $a_{2}=-\delta\kappa^{\#}\kappa\pi^{2-2e}$ and $ a_{3}=\delta\kappa  $. Write $R_i=\ord(a_i)$. First note that $  \prec a_{1},a_{2},a_{3} \succ $ exists. Indeed, $ \kappa^{\#} $ is a unit and $ d(\kappa^{\#})=1 $, by Proposition\;\ref{prop:duality2}. Hence $ R_{1}=R_{3}=\ord\, a_{1}=\ord\, a_{3}=0 $ and $ R_{2}=\ord\, a_{2}=2-2e $. Also, $ d(-a_{1}a_{2})=d(\kappa)=2e-1 $ and $ d(-a_{2}a_{3})=d(\kappa^{\#})=1 $. An easy verification using Lemma \ref{lem:2.2} shows that the lattice $ L:=\prec a_{1},a_{2},a_{3} \succ$ is well defined.

  Let $  N=2^{-1}\pi A(2,2\rho)\perp \langle\Delta\delta\rangle $ and   put $ S_{i}=R_{i}(N) $. To show that $ N\cong L $, we start by proving $S_i=R_i$ for $i=1,\,2,\,3$.

	    By the algorithm from \cite[\S 7]{beli_representations_2006}, we have a maximal norm splitting $ N=N_{1}\perp N_{2} $, where $ N_{1} $ (resp. $ N_{2} $) is binary (resp. unary), $ \mathfrak{s}(N_{i})=\mathfrak{p}^{r_{i}} $, $ \mathfrak{n}(N_{i})=\mathfrak{p}^{v_{i}} $ and $ \mathfrak{n}(N^{\mathfrak{s}(N_{i})})=\mathfrak{p}^{u_{i}} $. We have $ r_{1}=1-e $, $ v_{1}=1$ and $ r_{2}=v_{2}=0 $, so
	    \begin{align*}
	    		   u_{1}=\min\{v_{1},v_{2}\}=0\quad\text{and}\quad u_{2}=\min\{v_{2},2(r_{2}-r_{1})+v_{1}\}=0\,.
	    \end{align*}
	    Then, by \cite[Lemma 3.3(iii)]{beli_integral_2003}, the $ R $-invariants of $ N_{i} $ are given by
	    \begin{align*}
	    	&R_{1}(N_{1})=u_{1}=0\,,\quad  R_{2}(N_{1})=2r_{1}-u_{1}=2-2e\quad\text{and}\quad R_{1}(N_{2})=r_{2}=u_{2}=0\,.
	    \end{align*}
	    By \cite[Lemma 4.3(iii)]{beli_integral_2003},  the $ R $-invariants of $ N  $ can be obtained by putting together those of $ N_{1} $ and $ N_{2} $, so $ (S_{1},S_{2},S_{3})=(R_{1}(N_{1}),R_{2}(N_{1}),R_{1}(N_{2}))=(0,2-2e,0) $. This shows that $S_{i}=R_{i}$ for all $i=1,\,2,\,3$.

Now we apply \cite[Theorem 3.2]{beli_representations_2006} to prove that $N\cong L$. By \cite[Lemma 2.1]{beli_integral_2003}, $ \det FN=-\delta =\det FL $. Note that  Proposition \ref{prop:duality2} implies $ (\kappa,-\kappa\kappa^{\#})_{\p}=-1 $, so $ [1,-\kappa,\kappa\kappa^{\#}] $ is anisotropic. After scaling by $ \delta\kappa^{\#} $, we get that $ FL\cong [\delta\kappa^{\#},-\delta\kappa^{\#}\kappa,\delta\kappa] $ is anisotropic as well. We have shown that $ FN\cong [\pi,-\Delta\pi,\Delta\delta] $ is anisotropic in the proof of Proposition \ref{prop:space}(i). So $ FL\cong FN $ by \cite[63:20]{omeara_quadratic_1963}. In remains to check conditions (ii)--(iv) in \cite[Theorem 3.2]{beli_representations_2006}.

Let $\alpha_i=\alpha_i(L)$ and $\beta_i=\alpha_i(N)$. Since $R_{2}-R_{1}=2-2e$ and $R_{3}-R_{2}=2e-2$, we obtain $ \alpha_{1}=1 $ and $ \alpha_{2}=2e-1 $ by Proposition \ref{prop:Ralphaproperty2}(iv). Since $ R_{i}=S_{i} $ for $ 1\le i\le 3 $, the same argument gives $\beta_1=1$ and $\beta_2=2e-1$. So condition (ii) in \cite[Theorem 3.2]{beli_representations_2006} is verified. For condition (iii) of that theorem, note that, by \eqref{eq:alpha-defn}, $d[-a_{1,2}]\ge R_{1}-R_{2}+\alpha_{1}=2e-1$. Similarly, $ d[-b_{1,2}]\ge 2e-1 $. Hence
\begin{align*}
	d(a_{1,2}b_{1,2})\ge d[a_{1,2}b_{1,2}]\ge 2e-1=\alpha_{2}
\end{align*}
by the domination principle. Since $ \ord(a_{1}b_{1})=R_{1}+S_{1}=0 $ is even, $ d(a_{1}b_{1})\ge 1=\alpha_{1} $. Finally, since  $ \alpha_{1}+\alpha_{2}=2e $, there is no need to check condition (iv) in  \cite[Theorem 3.2]{beli_representations_2006}. We may thus conclude that $N\cong L$, as desired.
\end{proof}

	\begin{lem}\label{lem:BONGformaximallattice}
		Let $ k\ge 0  $ be an integer and $ L $ an integral $ \mathcal{O}_{F} $-lattice of rank $ \ell $. If $ L\cong \prec c_{1},\ldots,c_{\ell} \succ $ relative to a good BONG, then $ \mathbf{H}^{k}\perp L\cong \prec 1,-\pi^{-2e},\ldots, 1,-\pi^{-2e}, c_{1},\ldots,c_{\ell}\succ $ relative to a good BONG.
	\end{lem}
	\begin{proof}
	 By Lemma \ref{lem:2.2}, the lattice
$\prec 1,-\pi^{-2e},\ldots, 1,-\pi^{-2e},c_{1},\ldots,c_{\ell}\succ$ exists. Also,  by \eqref{eq:integralcondition}, we have $ -2e<0\le \ord(c_{1}) $. So we can apply \cite[Corollary 4.4(i)]{beli_integral_2003} to obtain
		\begin{align*}
			\mathbf{H}^{k}\perp L&\cong \prec 1,-\pi^{-2e}\succ \perp \ldots \perp \prec 1,-\pi^{-2e}\succ \perp L \cong   \prec 1,-\pi^{-2e},\ldots, 1,-\pi^{-2e},c_{1},\ldots,c_{\ell}\succ\,,
		\end{align*}
		  noticing that $\mathbf{H}\cong  \prec 1,\,-\pi^{-2e}\succ$ (Lemma\;\ref{lem3.10new} (i)).
	\end{proof}

Let $\tilde{N}_{\nu}(c)$ denote temporarily the lattices in the $N_{\nu}(c)$ column of the table in Proposition \ref{prop:maximallattices}.

	\begin{lem}\label{lem:maximal-BONG-odd}
		Let $ n\ge 2 $ and let $ N $ be an $ \mathcal{O}_{F} $-lattice of rank $ n $. Let $ S_{i}=R_{i}(N) $.
		
		(i) Suppose that $n$ is even.

If $ N=\tilde{N}_{1}^{n}(1) $ or $ \tilde{N}_{1}^{n}(\Delta) $, then $ S_{i}=0 $ for $ i\in [1,n]^{O} $ and $ S_{i}=-2e $ for $ i\in [1,n]^{E} $.
		
		If $ N=\tilde{N}_{2}^{n}(1)$ or $ \tilde{N}_{2}^{n}(\Delta) $, then $ S_{i}=0 $  for $ i\in [1,n-2]^{O} $, $ S_{i}=-2e $ for $ i\in [1,n-2]^{E} $,
		$ S_{n-1}=1 $ and $ S_{n}=1-2e $.
		
	If $ N=\tilde{N}_{1}^{n}(c) $ or $ \tilde{N}_{2}^{n}(c) $, with $ c\in F^{\times}/F^{\times 2} $ and $ c\not=1,\Delta $, then $ S_{i}=0 $ for $ i\in [1,n-1]^{O} $, $ S_{i}=-2e $ for $ i\in [1,n-1]^{E} $ and $ S_{n}=1-d(c)$.

		(ii) Suppose that $ n $ is odd and $ \delta\in \mathcal{O}_F^{\times} $.

 If $ N=\tilde{N}_{1}^{n}(\delta) $, then $ S_{i}=0 $ for $ i\in [1,n]^{O} $ and $ S_{i}=-2e $ for $ i\in [1,n]^{E} $.
		
	If $ N=\tilde{N}_{2}^{n}(\delta)$, then $ S_{i}=0 $  for $ i\in [1,n]^{O} $, $ S_{i}=-2e $ for $ i\in [1,n-2]^{E} $
	  and $ S_{n-1}=2-2e $.
		
	If $ N=\tilde{N}_{1}^{n}(\delta\pi) $ or $ \tilde{N}_{2}^{n}(\delta\pi) $, then $ S_{i}=0 $ for $ i\in [1,n-1]^{O} $, $ S_{i}=-2e $ for $ i\in [1,n-1]^{E} $ and $ S_{n}=1$. (Note that  $d(\delta\pi)=0$. Hence the formula $S_n=1-d(c)$ is still true for $N=\tilde{N}_{1}^{n}(c) $ or $ \tilde{N}_{2}^{n}(c) $, with $c=\delta\pi$, when $n$ is odd.)
	\end{lem}
	\begin{proof}
This follows easily by combining Lemmas\;\ref{lem3.10new} and \ref{lem:BONGformaximallattice} and \cite[Corollary 4.4(i)]{beli_integral_2003}.
\end{proof}

	  \begin{proof}[Proof of Proposition \ref{prop:maximallattices}]
Let $V=W_{\nu}^n(c)$, with $\nu\in \{1,\,2\}$. 	Recall from \cite[93:11]{omeara_quadratic_1963} that the  $\mathcal{O}_F$-maximal lattice in  $\mathbb{H}$ is $\mathbf{H}$. By \cite[82:23]{omeara_quadratic_1963}, if $V$ is isotropic, then the $\mathcal{O}_F$-maximal lattice in $V$ can be split as $\mathbf{H}\perp L$, where $L$ is the $\cO_F$-maximal lattice on $W^{n-2}_{\nu}(c)$. Thus this allows us to reduce to the case where $V$ is anisotropic of dimension $2\le n\le 4$. (So $c\neq 1$ if $n=2$.)

Let $M=N^n_{\nu}(c)$ be the $\cO_F$-maximal lattice on $V$ and let $N=\tilde{N}_{\nu}(c)$. Put $ R_{i}=R_{i}(M)$ and $S_i=R_i(N)$. We have $FN=V=FM$. So the maximality of $M$ implies $N\subseteq M$. Assume that the strict inclusion relation $ N\subset M$ holds. Then we have the strict inclusion relation  $\mathfrak{v}(N)\subset\mathfrak{v}(M)$ between the volumes of $N$ and $M$. Note that $\mathfrak{v}(L)=\det(L)\mathcal{O}_{F} $ for any lattice $ L $, so, by \cite[Lemma 2.1]{beli_integral_2003}, we have $ \ord(\mathfrak{v}(L))=\ord(\det L)=\sum_{i=1}^{n}R_{i}(L) $. Thus, by taking orders, we get $ \sum_{i=1}^{n}S_{i}>\sum_{i=1}^{n}R_{i} $. But, by \cite[82:11]{omeara_quadratic_1963}, $ \mathfrak{v}(M) $ and $\mathfrak{v}(N) $ differ by a square factor, so their orders have the same parity, i.e. $ \sum_{i=1}^{n}R_{i}\equiv\sum_{i=1}^{n}S_{i} \pmod{2}$. Hence $ \sum_{i=1}^{n}S_{i}>\sum_{i=1}^{n}R_{i} $ implies $ \sum_{i=1}^{n}R_{i}\le \sum_{i=1}^{n}S_{i}-2 $.

The values of $S_i$ have been determined in Lemma \ref{lem:maximal-BONG-odd}. Also, by Proposition \ref{prop:Ralphaproperty3}(i), we have $ R_{i}\ge R_{1}\ge 0 $ for odd $ i $ and $ R_{i}\ge R_{2}\ge -2e $ for even $ i $. It is now sufficient to show that the inequality $ \sum_{i=1}^{n}R_{i}\le \sum_{i=1}^{n}S_{i}-2 $ always leads to a contradiction. We do this case by case.

\medskip
\noindent {\bf Case I:} $n=2$.

Recall that $c\neq 1$ in this case.

If $ c=\Delta $ and $ \nu=1 $, then $ R_{1}+R_{2}\le S_{1}+S_{2}-2=0+(-2e)-2=-2e-2 $, which is impossible since $ R_{1}\ge 0 $ and $ R_{2}\ge -2e $.

If $ c=\Delta $ and $ \nu=2 $, then $ R_{1}+R_{2}\le S_{1}+S_{2}-2=1+(1-2e)-2=-2e $. But $ R_{1}\ge 0 $ and $ R_{2}\ge -2e $, so we must have $ R_{1}=0 $ and $ R_{2}=-2e $. By Proposition \ref{prop:Ralphaproperty3}(iv), we have $ FM\cong \mathbb{H}=W_{1}^{2}(1) $ or $ [1,-\Delta]=W_{1}^{2}(\Delta) $, which contradicts $ FM\cong W_{2}^{2}(\Delta) $.

If $ c\neq \Delta $, then $ R_{1}+R_{2}\le S_{1}+S_{2}-2=0+(1-d(c))-2=-1-d(c) $. But this is impossible, since $ R_{1}\ge 0 $ and, by \eqref{eq:BONGs}, $ R_{2}\ge R_{2}-R_{1}\ge -d(-a_{1}a_{2})=-d(c) $.

\noindent {\bf Case II:} $n=4$.

Since $ V $ is assumed to be anisotropic, we have $ c=\Delta $ and $ \nu=2 $. Then $ \sum_{i=1}^{4}R_{i}\le \sum_{i=1}^{4}S_{i}=0+(-2e)+1+(1-2e)-2=-4e $. Since $ R_{3}\ge R_{1}\ge 0 $ and $ R_{4}\ge R_{2}\ge -2e $, we get $ R_{1}=R_{3}=0 $ and $ R_{2}=R_{4}=-2e$. By Proposition \ref{prop:Ralphaproperty3}(iv), we have $ FM\cong \mathbb{H}^{2}=W_{1}^{4}(1) $ or $ \mathbb{H}\perp [1,-\Delta]=W_{1}^{4}(\Delta) $, which contradicts $ FM\cong W_{2}^{4}(\Delta) $.

\noindent {\bf Case III:} $n=3$.

Since $ V $ is anisotropic, we must have $ \nu=2 $. If $ c=\delta\pi $, with $ \delta\in \mathcal{U} $, then $ \sum_{i=1}^{3}R_{i}\le \sum_{i=1}^{3}S_{i}-2=0+(-2e)+1-2\le -2e-1$. But this is impossible, since $ R_{3}\ge R_{1}\ge 0 $ and $ R_{2}\ge -2e $.

If $ c=\delta $, then $ \sum_{i=1}^{3}R_{i}\le \sum_{i=1}^{3}S_{i}-2=0+(2-2e)+0-2=-2e$. Since $ R_{3}\ge R_{1}\ge 0 $ and $ R_{2}\ge -2e $, we must have $ R_{1}=R_{3}=0 $ and $ R_{2}=-2e $. By Proposition \ref{prop:Ralphaproperty3}(v), we have $ FM\cong \mathbb{H}\perp [\varepsilon]=W_{1}^{3}(\varepsilon) $ for some $ \varepsilon\in \mathcal{O}_{F}^{\times} $. But this contradicts $ FM\cong W_{2}^{3}(\delta) $.

\medskip

We have derived a contradiction in all cases.  This proof is thus complete.
	\end{proof}

	\begin{re}
			In the proof of Proposition \ref{prop:maximallattices}, we have used the BONG theory. Giving a proof using only the classical theory is also possible, at least for even $n$ (cf. \cite[Proposition 4.2]{hhx_indefinite_2021} when $ n=2 $). However,  as a nice application of Beli's theory, we think that the above method is also worthy noticing.
	\end{re}

		\begin{proof}[Proof of Theorem \ref{thm:nuniversaldyadic15theorem}]
			By Remark \ref{re:maximallattices}, the lattices in this theorem are precisely the lattices in Proposition \ref{prop:maximallattices}. So, by Proposition \ref{prop:maximallattices} and \cite[82:18]{omeara_quadratic_1963}, these lattices form a testing set for the $n$-universality.
			
			To prove that the set is minimal for the $n$-universality test, consider any lattice $N$ in the set. By Proposition \ref{prop:space}(iii), there is a unique space $V$ of dimension $n+2$ which does not represent $FN$ but represents all the other $n$-dimensional spaces. The $\mathcal{O}_F$-maximal lattice in $V$ does not represent $N$, but it represents all the other $n$-dimensional $\mathcal{O}_F$-maximal lattices. This completes the proof.
		\end{proof}

 The following two lemmas will be used in Sections \ref{sec:even-con} and \ref{sec:odd-con}.

\begin{lem}\label{lem:spacerep-criterion}
			Let $n\ge 2$ and let $W_1$ and $W_2$ be $n$-dimensional quadratic spaces with $\det W_1=\det W_2=D$ and $W_1\not\cong W_2$ (e.g. $W_1=W_1^n(c)$ and $W_2=W_2^n(c)$ for some  $ c\in F^{\times}/F^{\times2} $). Let $V$ be a quadratic space over $F$.

Suppose either $ \dim V=n+1 $, or $  \dim V=n+2  $ with $\det V=-D$.

Then $ V $ represents exactly one of $ W_{1}$ and $ W_{2}$.
\end{lem}
		\begin{proof}
First consider the case $\dim V=n+1$. Then $V$ represents an $n$-dimensional space $W$ if and only if $V\cong W\perp [\det V\det W]$, by \cite[63:21]{omeara_quadratic_1963}. Note that $W_1\perp [D\det V]$ and $W_2\perp [D\det V]$ are non-isometric spaces with the same determinant as $V$ (by Witt cancellation and the hypothesis). Since there are exactly two isometry classes of $(n+1)$-dimensional spaces with the fixed determinant, $V$ is isometric to precisely one of the two spaces $W_1\perp [D\det V]$ and $W_2\perp [D\det V]$, i.e., $V$ represents precisely one of $W_1$ and $W_2$.

If $\dim V=n+2$ and $\det V=-D$, then $V$ represents an $n$-dimensional space $W$ with $\det W=D$ if and only if $V\cong W\perp\mathbb{H}$, by \cite[63:21]{omeara_quadratic_1963}. So the result can be proved in the same way as in the previous case.
\end{proof}

		\begin{lem}\label{lem:spacerep-criterion-2}
			Let $ n\ge 2 $, $ \mu\in \{1,\Delta\} $ and $ \varepsilon\in \mathcal{O}_{F}^{\times} $.
			
			(i) If $ n $ is even, then $ W_{1}^{n+1}(\varepsilon) $ represents $ W_{1}^{n}(\mu) $ but does not represent $ W_{2}^{n}(\mu) $.
			
(When $n=2$ we ignore $W_2^n(1)$, as $W^2_2(1)$ is not defined.)

			(ii) If $ n $ is odd, then $ W_{1}^{n+1}(\mu) $ represents $ W_{1}^{n}(\varepsilon) $.
		\end{lem}
		\begin{proof}
			(i) By Lemma\;\ref{lem:spacerep-criterion}, it is sufficient to show that $W_1^{n+1}(\varepsilon)$ represents $W_1^n(\mu)$. For $\mu=1$, this is clear from Defintion\;\ref{defn:space}. For $\mu=\Delta$, by Witt cancellation, it suffices to prove that $W_1^3(\varepsilon)=\mathbb{H}\bot[\varepsilon]=[1,\,-1,\,\varepsilon]$ represents $W_1^2(\Delta)=[1,\,-\Delta]$. Indeed, since $(\Delta,\,\varepsilon)_{\p}=1$, $-\Delta$ is represented by $[-1,\,\varepsilon]$, whence  $ [1,-\Delta] \rep [1,\,-1,\,\varepsilon] $.

			(ii) This is because both $ \mathbb{H} $ and $ [1,-\Delta] $ represent $ \varepsilon $.
		\end{proof}

%
%
%
%
%

		\section{Characterization of $n$-universality for even $ n $}\label{sec:even-con}
		
		Throughout this section,  let $ n \ge 2$ be an even integer, $M$ an integral $ \mathcal{O}_{F} $-lattice of rank $ m\ge n+2 $, and suppose that
		$  M\cong \prec a_{1},\ldots,a_{m}\succ $  relative to some good BONG. Write $ R_{i}=R_{i}(M) $ for $ 1\le i\le m $ and $ \alpha_{i}=\alpha_{i}(M) $ for $ 1\le i\le m-1 $.
		Whenever a rank $n$ $\mathcal{O}_F$-lattice $N$ is considered, we assume that $N\cong \prec b_1,\ldots, b_n\succ$ relative to some good BONG and we denote by $S_i=R_i(N)$ and $\beta_i=\alpha_i(N)$ the associated invariants.

		\begin{thm}\label{thm:even-nuniversaldyadic}
			The lattice	 $ M $ is $n$-universal if and only if the space $ FM $ is $n$-universal and the following conditions hold:
			
			\begin{itemize}
				\item[ $ I_{1}^{E}(n) $:] $ R_{i}=0 $ for  $ i\in [1,n+1]^{O} $ and $ R_{i}=-2e $ for $ i\in [1,n]^{E} $.

				\item[$ I_{2}^{E}(n)$:] Either $\alpha_{n+1}=0 $, or $ \alpha_{n+1}=1 $ and $ d[-a_{n+1,n+2}]=1-R_{n+2} $.
				
				\item[$ I_{3}^{E}(n) $:] If $m\ge n+3$ and $ R_{n+3}-R_{n+2}>2e $, then $ R_{n+2}=-2e $; and if moreover $ n\ge 4 $, or $ n=2 $ and $ d(a_{1,4})=2e $, then $ R_{n+3}=1 $.
			\end{itemize}
		\end{thm}
		\begin{proof}
			It suffices to combine Theorem \ref{thm:beligeneral} with  Lemmas \ref{lem:con12even}, \ref{lem:con3even} and \ref{lem:con4even} below.
		\end{proof}

		\begin{lem}\label{lem:con12even}
			Suppose that $ FM $ is n-universal. Then the following conditions are equivalent:
			
			(i) Theorem $\ref{thm:beligeneral}(i)(ii)$ hold for all integral $ \mathcal{O}_{F} $-lattices $ N $ of rank $ n $.
			
			(ii) Theorem $\ref{thm:beligeneral}(i)(ii)$ hold for  $N=N_{1}^{n}(1)$, $ N_{1}^{n}(\Delta) $ (cf. Proposition \ref{prop:maximallattices}).
			
			(iii) $ M $ satisfies the condition $ I_{1}^{E}(n) $ in Theorem $\ref{thm:even-nuniversaldyadic}$.
		\end{lem}
		\begin{proof}
			
			\textbf{(i)$ \Rightarrow $(ii)}: It is trivial.

			\textbf{(ii)$ \Rightarrow $(iii)}: For every $ 1\le i\le n/2 $, we have $ R_{2i-1} \ge 0 $ and $ R_{2i} \ge -2e $ by Proposition \ref{prop:Ralphaproperty3}(i). On the other hand, by \cite[Lemma 4.6(i)]{beli_representations_2006}, we have $ R_{2i-1}+R_{2i}\le S_{2i-1}+S_{2i}=-2e $. Hence $ R_{2i-1}=R_{2i}+2e=0 $. It remains to show that $ R_{n+1}=0 $.
			
			Suppose $ R_{n+1}>0 $. Then $ R_{n+1}-S_{n}=R_{n+1}+2e>2e$. Thus, the assumption $ d[a_{1,n}b_{1,n}]\ge A_{n}$ implies  $ a_{1,n}b_{1,n}\in F^{\times 2}$ by \cite[Corollary 2.10]{beli_representations_2019} and so $ a_{1,n}=b_{1,n}=\det FN $ (in $ F^{\times}/F^{\times 2} $). But $ \det FN_1^n(1)=(-1)^{n/2}$, $  \det FN_1^n(\Delta)=(-1)^{n/2}\Delta $, and $  a_{1,n} $ cannot be both $ (-1)^{n/2} $ and $ (-1)^{n/2}\Delta $. So Theorem \ref{thm:beligeneral}(ii) fails for $ N=N_{1}^{n}(1) $ or $ N_{1}^{n}(\Delta) $.

			\textbf{(iii)$ \Rightarrow $(i)}: By Proposition \ref{prop:Ralphaproperty3}(i), we have $ S_{i}\ge 0$ for any odd $ i $ and $ S_{i}\ge -2e $ for any even $ i $. Hence, by $ I_{1}^{E}(n) $, Theorem \ref{thm:beligeneral}(i) holds and
			\begin{align}\label{eq:notessentialcondition}
				R_{i+1}\le S_{i-1}\quad \text{ for } 2\le i\le n\,.
			\end{align} So the indices $ 2,\ldots, n $ are not essential (in the sense of \cite[Definition 7]{beli_representations_2019}). By \cite[Lemma 2.12]{beli_representations_2019}, we only need to prove that $d[a_{1,i}b_{1,i}]\ge A_i$ for $ i=1 $ and $ i=n $. For $ i=1 $, since $ S_{1}\ge 0 $, $ A_{1}\le (R_{2}-S_{1})/2+e\le (-2e-0)/2+e\le 0\le d[a_{1}b_{1}]$. For $ i=n $, it suffices to apply  Lemma \ref{lem:Aj} with $ j=n $, since $ R_{n}=-2e $ and $ R_{n+1}=0 $.
	\end{proof}

		\begin{lem}\label{lem:evenlatticeproperty2}
	  Suppose that $ M $ satisfies $ I_{1}^{E}(n) $. If $  R_{n+2}\ge 2-2e $ and $ d[-a_{n+1,n+2}]>1-R_{n+2} $, then
\[
R_{n+2}>S_n\quad\text{and}\quad d[-a_{1,n+1}b_{1,n-1}]+d[-a_{1,n+2}b_{1,n}]>2e+S_n-R_{n+2}\,,
\]but $[b_1,\ldots, b_n]$ is not represented by $[a_1,\ldots, a_{n+1}]$; in other words, Theorem \ref{thm:beligeneral}(iii) fails at $ i=n+1 $ for $ N=N_{2}^{n}(\Delta) $ (cf. Proposition \ref{prop:maximallattices}).
	\end{lem}
	\begin{proof}

		 First, $ [b_{1},\ldots,b_{n}]=FN=W_{2}^{n}(\Delta) $ by Definition \ref{defn:maximallattices}, and $ [a_{1},\ldots,a_{n+1}]\cong W_{1}^{n+1}(\varepsilon) $ for some $ \varepsilon\in \mathcal{O}_{F}^{\times} $ by Proposition \ref{prop:Ralphaproperty3}(v). Hence $ [b_{1},\ldots,b_{n}]$ is not represented by $[a_{1},\ldots,a_{n+1}] $ by Lemma \ref{lem:spacerep-criterion-2}(i).

		We have  $S_{n}=1-2e$ by Lemma\;\ref{lem:maximal-BONG-odd}(i), so $R_{n+2}\ge 2-2e>S_{n}$.
		
		Since $ R_{n}=-2e $, by Proposition \ref{prop:Ralphaproperty3}(iii), we have $ d[(-1)^{n/2}a_{1,n}]\ge 2e>1-R_{n+2} $. Also, $ d[-a_{n+1,n+2}]>1-R_{n+2} $ by the hypothesis. Hence $ d[(-1)^{(n+2)/2}a_{1,n+2}]>1-R_{n+2} $ by the domination principle. On the other hand, in $ F^{\times}/F^{\times 2} $ we have $ b_{1,n}=\det FN=\det W_{2}^{n}(\Delta)=(-1)^{n/2}\Delta $, so $ d[(-1)^{n/2}b_{1,n}]=d((-1)^{n/2}b_{1,n})=d(\Delta)=2e>1-R_{n+2} $. By the domination principle, we get $ d[-a_{1,n+2}b_{1,n}]>1-R_{n+2} $. Also, by Proposition \ref{prop:Ralphaproperty2}(i), we have $ d[-a_{1,n+1}b_{1,n-1}]=\beta_{n-1}=0 $. So		$d[-a_{1,n+1}b_{1,n-1}]+d[-a_{1,n+2}b_{1,n}]>0+(1-R_{n+2})=2e+S_{n}-R_{n+2}$. This completes the proof.
	\end{proof}

		\begin{lem}\label{lem:con3even}
			Suppose that $ FM $ is n-universal and $M$ satisfies the condition $ I_{1}^{E}(n) $ in Theorem $\ref{thm:even-nuniversaldyadic}$. Then the following conditions are equivalent:
			
			(i) Theorem $\ref{thm:beligeneral}(iii)$ holds for all integral $ \mathcal{O}_{F} $-lattices $ N $ of rank $ n $.
			
			(ii) Theorem $\ref{thm:beligeneral}(iii)$ holds for $N=N_{2}^{n}(\Delta)$ (cf. Proposition \ref{prop:maximallattices}).
			
			(iii) $ M $ satisfies the condition $ I_{2}^{E}(n) $ in Theorem $\ref{thm:even-nuniversaldyadic}$.
		\end{lem}
		\begin{proof}	\textbf{(i)$ \Rightarrow $(ii)}: It is trivial.
			
			\textbf{(ii)$ \Rightarrow $(iii)}: If $ R_{n+2}-R_{n+1}=-2e $, then $ \alpha_{n+1}=0 $ by Proposition \ref{prop:Ralphaproperty2}(i). Hence we may assume that $ R_{n+2}-R_{n+1}>-2e $. Then $ R_{n+2}-R_{n+1}\ge 2-2e $ by Corollary \ref{cor:R-R-odd}(i). Thus $ R_{n+2}\ge 2-2e $, since $ R_{n+1}=0 $ (by $I^E_1(n)$). By Lemma \ref{lem:d[-ai+1ai+2]=1-Ri+2}(i), we have $ d[-a_{n+1}a_{n+2}]\ge 1-R_{n+2} $. If $ d[-a_{n+1}a_{n+2}]> 1-R_{n+2} $, from Lemma \ref{lem:evenlatticeproperty2} we see that Theorem \ref{thm:beligeneral}(iii) fails at $ i=n+1 $ for $ N=N_{2}^{n}(\Delta) $. This contradicts condition (ii). Hence $ d[-a_{n+1,n+2}]=1-R_{n+2} $. So $ \alpha_{n+1}=1 $ by  Lemma \ref{lem:d[-ai+1ai+2]=1-Ri+2}(i).

			\textbf{(iii)$ \Rightarrow $(i)}:	In view of \eqref{eq:notessentialcondition}, we only need to consider Theorem \ref{thm:beligeneral}(iii) for  $ i=n+1= \min\{m-1,n+1\}$.
			
			If $ S_{n}=-2e $, then $ S_{i}=0 $ for $ i\in [1,n]^{O} $ and $ S_{i}=-2e $ for $ i\in [1,n]^{E} $ by Proposition \ref{prop:Ralphaproperty3}(iii), so $ [b_{1},\ldots,b_{n}]\cong W_{1}^{n}(1) $ or $ W_{1}^{n}(\Delta) $ by Proposition \ref{prop:Ralphaproperty3}(iv). Also, since $ R_{n+1}=R_{n}+2e =0$, $ [a_{1},\ldots,a_{n+1}]\cong W_{1}^{n+1}(\varepsilon) $ for some $ \varepsilon\in \mathcal{O}_{F}^{\times} $ by Proposition \ref{prop:Ralphaproperty3}(v). In both cases, $ [b_{1},\ldots,b_{n}] $ is represented by $ [a_{1},\ldots,a_{n+1}] $ by Lemma \ref{lem:spacerep-criterion-2}(i).

			Now suppose $ S_{n}\ge 1-2e $. To check Theorem\;\ref{thm:beligeneral}(iii) we further assume that $ R_{n+2}>S_{n} $ and $ d[-a_{1,n+1}b_{1,n-1}]+d[-a_{1,n+2}b_{1,n}]>2e+S_{n}-R_{n+2} $. If $ d[-a_{1,n+2}b_{1,n}]\not=1-R_{n+2} $, then $ R_{n+2}\le S_{j}\le S_{n} $ for some $ j\in [1,n]^{E} $ by Lemma \ref{lem:d[-ai+1ai+2]=1-Ri+2}(iii) and \eqref{eq:GoodBONGs}, a contradiction. Hence $ d[-a_{1,n+2}b_{1,n}]=1-R_{n+2} $. Note that $R_{n+2}-R_{n+1}=R_{n+2}>S_n\ge 1-2e$. So $\alpha_{n+1}\neq 0$ by Proposition\;\ref{prop:Ralphaproperty2}(i). Thus $\alpha_{n+1}=1$ by $I_2^E(n)$. Since $d[-a_{1,n+1}b_{1,n-1}]\le  \alpha_{n+1}=1 $, we have
			\begin{align*}
			2-R_{n+2}=1+(1-R_{n+2})\ge d[-a_{1,n+1}b_{1,n-1}]+d[-a_{1,n+2}b_{1,n}]>2e+S_{n}-R_{n+2}\,,
			\end{align*}
			which implies $ S_{n}<2-2e $. Since $ S_{n-1}\ge 0$ by Proposition \ref{prop:Ralphaproperty3}(i), it follows that $ -2e\le S_{n}-S_{n-1}\le S_{n}<2-2e $. So $ S_{n}-S_{n-1}=-2e $ by Corollary \ref{cor:R-R-odd}(i) and hence $ d[-a_{1,n+1}b_{1,n-1}]=\beta_{n-1}=0 $ 	by Proposition \ref{prop:Ralphaproperty2}(i). Therefore, we deduce that
			\begin{align*}
				1-R_{n+2}=0+(1-R_{n+2})=d[-a_{1,n+1}b_{1,n-1}]+d[-a_{1,n+2}b_{1,n}]>2e+S_{n}-R_{n+2}\,,
			\end{align*}
		 which implies $ S_{n}<1-2e $. This contradicts the assumption $ S_{n}\ge 1-2e $.
		\end{proof}

		\begin{lem}\label{lem:con4even}
			Suppose that $ FM $ is n-universal and $M$ satisfies the conditions $I_1^{E}(n)$ and $I_{2}^{E}(n) $ in Theorem $\ref{thm:even-nuniversaldyadic}$. Then the following conditions are equivalent:
			
			(i)  Theorem $\ref{thm:beligeneral}(iv)$ holds for all integral $ \mathcal{O}_{F} $-lattices $ N $ of rank $ n $.
			
			(ii) Theorem $\ref{thm:beligeneral}(iv)$ holds for all the lattices $N$ in the following list if $m\ge n+3$ and $ R_{n+3}-R_{n+2}>2e $:
			\begin{align*}
				N_{2}^{n}(1) \,\text{(if $ n\ge 4 $)},\; N_{2}^{n}(\Delta)\;\text{and}\; N_{\nu}^{n}(c)\,,\;\;\text{ with }  \nu \in \{1,2\} \,,\, c\in F^{\times}/F^{\times 2} \text{ and } d(c)<2e
			\end{align*}
		(cf. Proposition \ref{prop:maximallattices}).
			
			(iii) $ M $ satisfies the condition $ I_{3}^{E}(n) $ in Theorem $\ref{thm:even-nuniversaldyadic}$. 	
		\end{lem}
		\begin{proof} 	\textbf{(i)$ \Rightarrow $(ii)}: It is trivial.

			\textbf{(ii)$ \Rightarrow $(iii)}: Assume that $ R_{n+3}-R_{n+2}>2e $. This implies $\alpha_{n+2}>2e$ by Proposition \ref{prop:Ralphaproperty2}(ii).

We claim that $ \alpha_{n+1}=0 $. If not, then $ \alpha_{n+1}=1 $ and $ d[-a_{n+1}a_{n+2}]=1-R_{n+2} $ by $ I_{2}^{E}(n) $. It follows  from Lemma \ref{lem:d[-ai+1ai+2]=1-Ri+2}(i) that $ R_{n+2}\ge 2-2e $. So, by Lemma \ref{lem:d[-ai+1ai+2]=1-Ri+2}(ii), we have $d[(-1)^{(n+2)/2}a_{1,n+2}]=1-R_{n+2}$. But
			\[
			d[(-1)^{(n+2)/2}a_{1,n+2}]=\min\{d((-1)^{(n+2)/2}a_{1,n+2}),\alpha_{n+2}\}
			\]and $ \alpha_{n+2}>2e>1-R_{n+2} $. It follows that			$ d((-1)^{(n+2)/2}a_{1,n+2})=1-R_{n+2}$.

		   	Write $V=[a_{1},\ldots,a_{n+2}] $. Let $ N=N_{\nu}^{n}(c) $, with $ \nu\in \{1,2\} $ and  $ c=(-1)^{(n+2)/2}a_{1,n+2} $. Then $ \det V=a_{1,n+2}=(-1)^{(n+2)/2}c=-\det W_{\nu}^{n}(c)=-\det FN$. Since $ d(c)=1-R_{n+2}<2e $, we have $ S_{n}=1-d(c)=R_{n+2} $ by Lemma \ref{lem:maximal-BONG-odd}(i). Now $ R_{n+3}>R_{n+2}+2e=S_{n}+2e $, so $ FN=[b_{1}\ldots,b_{n}]\rep [a_{1},\ldots,a_{n+2}]=V$ by condition (ii). This shows that $V$ represents both $W^n_1(c)=FN_1^n(c)$ and $W^n_2(c)=FN_2^n(c)$. But this contradicts  Lemma \ref{lem:spacerep-criterion}. So the claim is proved.

		Now $ \alpha_{n+1}=0 $.   So Proposition \ref{prop:Ralphaproperty2}(i) implies $ R_{n+2}=R_{n+2}-R_{n+1}=-2e $. Since $ R_{n+3}-R_{n+2}>2e $, we have $ R_{n+3}\ge 1 $. Assume that $ I_{3}^{E}(n) $ fails. Then $ R_{n+3}>1 $ and either $ n\ge 4 $, or $ n=2 $ and $ d(a_{1,4})=2e $. Since $ R_{n+2}=-2e $, we have $ [a_{1},\ldots,a_{n+2}]\cong W_{1}^{n+2}(\mu) $, with $ \mu\in \{1,\Delta\} $, by Proposition \ref{prop:Ralphaproperty3}(iv). Also, if $ n=2 $, then $ d(a_{1,4})=2e$, so $ \mu=\Delta$ in this case.
			
			Take $ N=N_{2}^{n}(\mu) $, which is always defined since for $n=2$ we have $\mu=\Delta$. Then $ S_{n}=1-2e $ by Lemma \ref{lem:maximal-BONG-odd}(i). Recall that $ R_{n+2}=-2e $ and $ R_{n+3}>1 $, so the condition $R_{n+3}>S_{n}+2e\ge R_{n+2}+2e$ is fulfilled. Hence $ [b_{1},\ldots,b_{n}]\rep [a_{1},\ldots,a_{n+2}]$ by condition (ii). But $ [a_{1},\ldots,a_{n+2}]\cong W_{1}^{n+2}(\mu) $, $ [b_{1},\ldots,b_{n}]\cong  FN=W_{2}^{n}(\mu) $, and $ W_{1}^{n+2}(\mu) $ does not represent $ W_{2}^{n}(\mu) $ by Proposition \ref{prop:space}(iii). A contradiction is derived.
			
			\textbf{(iii)$ \Rightarrow $(i)}:  For $ 1<i\le n-1$, we have  $ R_{i+2}-R_{i+1}\le 2e $ by $ I_{1}^{E}(n) $. Since $\alpha_{n+1}\le 2e$ by $ I_{2}^{E}(n) $, we have $ R_{n+2}-R_{n+1}\le 2e $ by Proposition \ref{prop:Ralphaproperty2}(ii). Hence we may suppose $ m\ge n+3 $ and only need to consider Theorem \ref{thm:beligeneral}(iv) for $ i=n+1 $. Thus, we assume that $R_{n+3}>S_n+2e\ge R_{n+2}+2e$ and we want to show that $[b_1,\ldots, b_n]\rep [a_1,\ldots, a_{n+2}]$.
		
		Since $ R_{n+3}-R_{n+2}>2e $, we have $ R_{n+2}=-2e $ by the first part of $ I_{3}^{E}(n) $.
			It follows that $ [a_{1},\ldots,a_{n+2}]\cong  W_{1}^{n+2}(1)$ or $ W_{1}^{n+2}(\Delta) $ by Proposition \ref{prop:Ralphaproperty3}(iv).
			
			If $ n=2 $, then since $ R_{4}=-2e $ (by $ I_{1}^{E}(2) $), we have $ d(a_{1,4})\ge d[a_{1,4}]\ge 2e $ by Proposition \ref{prop:Ralphaproperty3}(iii). If $ d(a_{1,4})=\infty$, then the space  $ [a_{1},a_{2},a_{3},a_{4}]\cong \mathbb{H}\perp \mathbb{H}=W^4_1(1)$ is $ 2 $-universal by \cite[Theorem 2.3]{hhx_indefinite_2021} and thus represents all binary quadratic spaces. So, if $n=2$ we may assume $ d(a_{1,4})=2e $ and hence $[a_1,a_{2},a_{3}, a_4]\cong W^4_1(\Delta)$. Then, by $ I_{3}^{E}(n) $, we get   $ R_{n+3}=1 $, both when $ n=2 $ and when $ n\ge 4 $. Hence $R_{n+3}=1>S_n+2e\ge R_{n+2}+2e=0$ and so $ S_{n}=-2e $. It follows that $ [b_{1},\ldots,b_{n}]\cong W_{1}^{n}(1)$ or $ W_{1}^{n}(\Delta) $ by Proposition \ref{prop:Ralphaproperty3}(iv). In all cases, $ [b_{1},\ldots,b_{n}]\rep [a_{1},\ldots,a_{n+2}] $ by Proposition \ref{prop:space}(iii).
		\end{proof} 	
	
		In the case $m=n+2=4$ we can derive the  following corollary from Theorem\;\ref{thm:even-nuniversaldyadic}.
		
		\begin{cor}\label{cor:even-nuniversaldyadic-quaternary}
			Let $ M $ be a quaternary integral $ \mathcal{O}_{F} $-lattice. Then the following conditions are equivalent:
			
			(i) $ M $ is $ 2 $-universal.
			
			(ii) $ FM\cong \mathbb{H}^{2} $, $ R_{1}=R_{3}=0$ and $R_{2}=R_{4}=-2e $.
			
			(iii) $M\cong 2^{-1}A(0,0)\perp 2^{-1}A(0,0) $.
			
		\end{cor}
		\begin{proof}
			\textbf{(i)$ \Rightarrow $(ii):} Suppose that $M$ is 2-universal. Then the space $FM$ is 2-universal. Hence $ FM\cong \mathbb{H}^{2}$  by \cite[Theorem 2.3]{hhx_indefinite_2021}. Moreover, $M$ satisfies $I_1^E(2)$ and $I_2^E(2)$ by Theorem \ref{thm:even-nuniversaldyadic}. From $I_1^E(2)$ we get $ R_{1}=R_{3}=0 $ and $ R_{2}=-2e $. If $ \alpha_{3}=1$, then $ d[-a_{3,4}]=1-R_{4} $ by $I^E_2(2)$. Lemma \ref{lem:d[-ai+1ai+2]=1-Ri+2}(ii) implies that $ d[a_{1,4}]=1-R_{4} $. But $ a_{1,4}\in F^{\times 2} $, so $ d[a_{1,4}]=d(a_{1,4})=\infty $, a contradiction. Hence $ \alpha_{3}=0 $ and so $ R_{4}=R_{4}-R_{3}=-2e $ (Proposition \ref{prop:Ralphaproperty2}(i)).

			\textbf{(ii)$ \Rightarrow $(i):} Suppose that (ii) holds. Then the space $ FM\cong \mathbb{H}^{2}$  is $ 2 $-universal, and we have $ \alpha_{3}=0 $  by Proposition \ref{prop:Ralphaproperty2}(i). Hence $M$ is 2-universal by Theorem\;\ref{thm:even-nuniversaldyadic}.		
			
			\textbf{(ii)$ \Rightarrow $(iii):}  Since $ R_{2}-R_{1}=-2e $, Corollary \ref{cor:R-R-odd}(ii) implies that $ \prec a_{1},a_{2} \succ \cong 2^{-1}A(0,0)$ or $ 2^{-1}A(2,2\rho) $. Similarly, $ \prec a_{3},a_{4} \succ \cong 2^{-1}A(0,0)$ or $ 2^{-1}A(2,2\rho) $. By \cite[Corollary 4.4(i)]{beli_integral_2003}, $ M\cong \prec a_{1},a_{2},a_{3},a_{4}\succ \cong \prec a_{1},a_{2}\succ \perp \prec a_{3},a_{4}\succ $. But $FM\cong \mathbb{H}^{2} $, so $\det FM=1$. Hence $ M\cong 2^{-1}A(0,0)\perp 2^{-1}A(0,0) $ or $ 2^{-1}A(2,2\rho)\perp 2^{-1}A(2,2\rho)$. By \cite[93:9 and 93:18(vi)]{omeara_quadratic_1963}, we have
			$2^{-1}A(2,2\rho)\perp 2^{-1}A(2,2\rho)\cong 2^{-1}A(0,0)\perp 2^{-1}A(0,0)$. So we get the desired result.
			
			\textbf{(iii)$ \Rightarrow $(ii):}	Using Lemma\;\ref{lem:maximal-BONG-odd}(i), we can get  the $R_i$ invariants of $\mathbf{H}^2=2^{-1}A(0,0)\perp 2^{-1}A(0,0) $. So the result follows easily.	
		\end{proof}

		The equivalence between (i) and (iii) in Corollary\;\ref{cor:even-nuniversaldyadic-quaternary} recovers \cite[Proposition\;4.5]{hhx_indefinite_2021}.
		
		\medskip
		
		In general, we have the following  concise criterion for $n$-universality (when $n\ge 2$ is even).
		
			\begin{thm}\label{thm:2universaldyadicver2}
			Let $ M\cong \prec a_{1},\ldots,a_{m}\succ $ be an integral $ \mathcal{O}_{F} $-lattice relative to some good BONG, $ R_{i}=R_{i}(M)$ for $ 1\le i\le m $ and $ \alpha_{i}=\alpha_{i}(M) $ for $ 1\le i\le m-1 $.
			
			Then $ M $ is $n$-universal if and only if $ m\ge n+3 $ or $m=n+2=4 $ and the following conditions hold:
			
			\begin{enumerate}
				\item[(i)] $ R_{i}=0 $ for  $ i\in [1,n+1]^{O} $ and $ R_{i}=-2e $ for $ i\in [1,n]^{E} $.
				\item[(ii)] If $ m=n+2=4 $, then $ FM\cong \mathbb{H}^{2} $ and $ R_{4}=-2e $.
				\item[(iii)] If $ m\ge n+3 $, then one has:
				
				\begin{enumerate}
					\item[(1)]  $ \alpha_{n+1}\le 1 $.
					\item[(2)] If $ R_{n+3}-R_{n+2}>2e  $, then $ R_{n+2}=-2e $; and if moreover either $n\ge 4$, or $n=2$ and $ d(a_{1,4})=2e $, then $ R_{n+3}=1 $.
					\item[(3)]  If $ R_{n+3}-R_{n+2}=2e $ and $ R_{n+2}=2-2e $, then $ d(-a_{n+1}a_{n+2})=2e-1 $.
				\end{enumerate}
			\end{enumerate}	
		\end{thm}
		
			\begin{proof}
			A necessary condition for $M$ to be $n$-universal is that the space $FM$ is $n$-universal, and the latter condition implies that either $m\ge n+3$, or $m=n+2=4$ and condition $FM\cong \mathbb{H}^{2}$ from (ii) holds, by \cite[Theorem 2.3]{hhx_indefinite_2021}.
			In view of Corollary \ref{cor:even-nuniversaldyadic-quaternary} we may assume that $m\ge n+3$. To prove the theorem, we compare the above conditions (i) and (iii) with the conditions $I^E_1(n)$,  $I_2^E(n)$ and  $ I^{E}_{3}(n)$ in  Theorem\;\ref{thm:even-nuniversaldyadic}.
			
			Condition (i) is the same as $I^E_1(n)$,  $I_2^E(n)$ implies (iii)(1), and (iii)(2) is equivalent to $ I^{E}_{3}(n)$. We may therefore assume that the conditions (i), (iii)(1) and (iii)(2) hold. By Proposition \ref{prop:Ralphaproperty2}(i),  we have $\alpha_{n+1}\in \{0,\,1\}$, and if $\alpha_{n+1}=0$, then $I_2^E(n)$ holds and $R_{n+2}=R_{n+1}-2e=-2e\neq 2-2e$. So we may assume that $\alpha_{n+1}=1$.
			
			Now we only need to show that under the above assumptions, (iii)(3) holds if and only if $d[-a_{n+1,n+2}]=1-R_{n+2}$.
			
			By Proposition \ref{prop:Ralphaproperty2}(vi), we have $d[-a_{n+1,n+2}]\ge 1-R_{n+2}$, and if $R_{n+2}\neq 2-2e$, then $d[-a_{n+1,n+2}]=1-R_{n+2}$. We may thus assume further that $ R_{n+2}=2-2e $. Then, by (iii)(2), we have $R_{n+3}-R_{n+2}\le 2e$.
			If $R_{n+3}-R_{n+2}<2e$, then $\alpha_{n+2}<2e$ by Proposition \ref{prop:Ralphaproperty2}(ii) and $\alpha_{n+2}\in \Z$ by Proposition \ref{prop:Ralphaproperty2}(i). So $d[-a_{n+1,n+2}]\le \alpha_{n+2}\le 2e-1=1-R_{n+2}$. As we have seen that $d[-a_{n+1,n+2}]\ge 1-R_{n+2}$, we get $d[-a_{n+1,n+2}]=1-R_{n+2}$ when  $R_{n+3}-R_{n+2}<2e$. Now consider the case $R_{n+3}-R_{n+2}=2e$. It remains to prove that $ d[-a_{n+1,n+2}]=2e-1 $ if and only if $ d(-a_{n+1,n+2})=2e-1 $. Recall that $R_{n+1}-R_{n}=R_{n+3}-R_{n+2}=2e$. By Proposition \ref{prop:Ralphaproperty2}(ii), we have $ \alpha_{n}=\alpha_{n+2}=2e$. But $ d[-a_{n+1,n+2}]=\min\{d(-a_{n+1,n+2}),\alpha_{n},\alpha_{n+2}\} $. This implies the desired equivalence and the theorem is thus proved.				
		\end{proof}

		\section{Characterization of $n$-universality for odd $ n $}\label{sec:odd-con}
		
		In  this section, we assume that $ n\ge 3 $ is an odd integer. By \cite[Theorem 2.1]{hhx_indefinite_2021}, a quadratic space over $F$ is $ n $-universal if and only if its dimension is at least $ n+3$. Let $ m\ge n+3$. As in the previous section, we fix a lattice $M$ and assume that $M\cong \prec a_{1},\ldots,a_{m}\succ $ relative to a good BONG. Invariants associated to the given BONG of $M$ will be denoted by the same notations as before, and similarly for any rank $n$  lattice $N$.
		
	\begin{thm}\label{thm:odd-nuniversaldyadic-2}\footnote{We are grateful to the anonymous referee for suggesting this theorem as an improved version of Proposition\;\ref{thm:odd-nuniversaldyadic}.}\label{cor:odd-nuniversaldyadic-simplify}
	The lattice $ M $ is $n$-universal if and only if $ m\ge n+3 $ (or equivalently, $FM$ is $n$-universal) and  the following conditions hold:
	\begin{itemize}
		\item[ $ I_{1}^{O}(n) $:]  $ R_{i}=0 $ for $ i\in [1,n]^{O} $, $ R_{i}=-2e $ for $ i\in [1,n]^{E} $, and  $ \alpha_{n}=0 $ or $ 1 $.
		
		\item[$ I_{2}^{O}(n) $:] If $ \alpha_{n}=0 $, then $ R_{n+2} \in \{0,1\}$.
		
		\item[ ] If $ \alpha_{n}=1 $ and either $ R_{n+1}=1 $ or $ R_{n+2}>1 $, then $ \alpha_{n+2}\le G_{n} $, where
\begin{equation}\label{eq5.1new}
\begin{split}
  G_n:&= 2(e-\lfloor(R_{n+2}-R_{n+1})/2\rfloor)-1\\
  &=\begin{cases}
		 		2e-R_{n+2}+R_{n+1}-1 &\text{if $ R_{n+2}-R_{n+1} $ is even}\,, \\
		 		2e-R_{n+2}+R_{n+1} &\text{if $ R_{n+2}-R_{n+1} $ is odd}\,.
		 	\end{cases}
\end{split}
\end{equation}
		\item[$ I_{3}^{O}(n) $:] $ R_{n+3}-R_{n+2}\le 2e $.
	\end{itemize}
\end{thm}	
\begin{proof}We will show in
Proposition \ref{thm:odd-nuniversaldyadic} that $M$ is $n$-universal if and only if $FM$ is $n$-universal and $ M $ satisfies $I_1^{E}(n-1),\, I_{2}^{E}(n-1),\, I_{3}^{E}(n-1) $  (cf. Theorem $\ref{thm:even-nuniversaldyadic}$), $ I_{2}^{O}(n) $ and $ I_{3}^{O}(n) $. Clearly, $ I_{1}^{O}(n) $ holds if and only if  $ I_{1}^{E}(n-1) $ holds and  $ \alpha_{n}=0$ or $1 $. So $I_1^{E}(n-1)$ and $I_{2}^{E}(n-1)$ together imply $ I_{1}^{O}(n) $.

Assume that $M$ satisfies $I_{1}^{O}(n)$ and $ I_{2}^{O}(n) $. Then $ I_{2}^{E}(n-1) $ follows from Lemma \ref{lem:simplifythm}(ii) below. If $ \alpha_{n}=0 $, then $ R_{n+1}=-2e $ by Lemma \ref{lem:simplifythm}(i), and $R_{n+2}\in\{0, 1\}$ by $I_{2}^{O}(n)$. In this case we have either $R_{n+2}-R_{n+1}=2e$ or $R_{n+2}=1$. If $\alpha_n\neq 0$, then $\alpha_n=1$ by $I_{1}^{O}(n)$, and thus $R_{n+2}-R_{n+1}\le 2e-1< 2e$ by Lemma \ref{lem:simplifythm}(ii). So we see that $I_{3}^{E}(n-1)$ holds.  This proves the theorem.
\end{proof}

		\begin{re}\label{re:Rn+1=Rn+2=1}
In fact, the following two conditions are equivalent:

(i) Either $ R_{n+1}=1 $ or $ R_{n+2}>1 $.

(ii) Either $ R_{n+1}=R_{n+2}=1 $ or $ R_{n+2}>1 $.

To see this, it suffices to show that $R_{n+1}=1$ implies $R_{n+2}\ge 1$.

			Suppose $ R_{n+1}=1 $ and $R_{n+2}< 1$. Since $ R_{n+2}\ge 0 $ (by Proposition\;\ref{prop:Ralphaproperty3}(i)), we get $ R_{n+2}=0 $. But then $R_{n+2}-R_{n+1}=-1$, which contradicts  Corollary\;\ref{cor:R-R-odd}(i).
		\end{re}

	Theorem  \ref{cor:odd-nuniversaldyadic-simplify} can be easily rephrased in the style of \cite[Theorem 2.1]{beli_universal_2020} as the following.
	
	\begin{thm}\label{thm:odd-nuniversaldyadic-beli-style}
		The lattice $ M $ is $n$-universal if and only if $ m\ge n+3 $, $ R_{i}=0 $ for $ i\in [1,n]^{O} $, $ R_{i}=-2e $ for $ i\in [1,n]^{E} $, $ R_{n+3}-R_{n+2}\le 2e $ and one of the following conditions holds:		
	
	(i) $ \alpha_{n}=0 $ (or equivalently, $ R_{n+1}=-2e $) and $ R_{n+2}\le 1 $ (or equivalently, $ R_{n+2}\in \{0,1\} $).
			
	(ii) $ \alpha_{n}=1 $ and, if either $ R_{n+1}=1 $ or $ R_{n+2}>1 $, then $ \alpha_{n+2}\le 2(e-\lfloor(R_{n+2}-R_{n+1})/2\rfloor)-1 $.
	\end{thm}

\begin{lem}\label{lem:simplifythm}
	Suppose that $ M $ satisfies $ I_{1}^{E}(n-1) $ and $ I_{2}^{O}(n) $.
	
	(i) If $ \alpha_{n}=0 $, then $ R_{n+1}=-2e $.
	
	(ii) If $ \alpha_{n}=1 $, then $ R_{n+2}-R_{n+1}\le 2e-1$ and $ d[-a_{n,n+1}]=1-R_{n+1} $. If moreover $ R_{n+1}=2-2e $, then $ \alpha_{n+1}=2e-1 $.
\end{lem}
\begin{proof}
	 	 (i)  This follows immediately from Proposition \ref{prop:Ralphaproperty2}(i), since $ R_{n}=0 $ by  $ I_{1}^{E}(n-1) $.
	 	
	 	 (ii)  Suppose $ \alpha_{n}=1 $. We have $ R_{n+1} \in [2-2e,0]^{E}\cup \{1\} $ by Proposition \ref{prop:Ralphaproperty2}(vi). If $ R_{n+2}-R_{n+1}\ge 2e $, then,  by $ I_{2}^{O}(n) $, we have
 \[
  \alpha_{n+2}\le G_{n}=2(e-\lfloor(R_{n+2}-R_{n+1})/2\rfloor)-1\le 2(e-\lfloor 2e/2\rfloor)-1=-1\,.
  \] This contradicts Proposition \ref{prop:Ralphaproperty2}(i). Hence we have $ R_{n+2}-R_{n+1}\le 2e-1 $.
	 	
	 	By Proposition \ref{prop:Ralphaproperty2}(vi), we have $ d[-a_{n,n+1}]\ge 1-R_{n+1} $, and equality holds except possibly when $ R_{n+1}=2-2e $. Assume that $ R_{n+1}=2-2e $. Then $ 2e-1=1-R_{n+1}\le d[-a_{n,n+1}]\le \alpha_{n+1}$. Recall that $R_{n+2}\ge 0$ from Proposition \ref{prop:Ralphaproperty3}(i). Now $  2e-2=0-(2-2e)\le R_{n+2}-R_{n+1}\le 2e-1 $, i.e. $ R_{n+2}-R_{n+1}\in \{2e-1,2e-2\} $, so $ \alpha_{n+1}=2e-1 $ by Proposition \ref{prop:Ralphaproperty2}(iii) and (iv). Hence $ 2e-1=1-R_{n+1}=d[-a_{n,n+1}]=\alpha_{n+1} $, as desired.
\end{proof}

		\begin{prop}\label{thm:odd-nuniversaldyadic}
			The lattice $ M $ is $n$-universal if and only if the space $FM $ is $n$-universal and $ M $ satisfies the conditions
$I_1^{E}(n-1),\, I_{2}^{E}(n-1),\, I_{3}^{E}(n-1) $, $ I_{2}^{O}(n) $ and $ I_{3}^{O}(n) $.
		\end{prop}
		\begin{proof}
			Note that if $ M $ is $ n $-universal, then it is also $ (n-1) $-universal. So, by Theorem \ref{thm:even-nuniversaldyadic}, $ M $ satisfies the conditions $ I_{1}^{E}(n-1) $, $ I_{2}^{E}(n-1) $ and $ I_{3}^{E}(n-1) $. For the remainder of the proof, we are done by combining Theorem \ref{thm:beligeneral} with  Lemmas \ref{lem:con12odd}, \ref{lem:con3odd} and \ref{lem:con4odd} below.
		\end{proof}

		\begin{lem}\label{lem:con12odd}
			Suppose that $ FM $ is $n$-universal.
			
			(i) If $ M $ satisfies  $ I_{1}^{E}(n-1) $, then Theorem $\ref{thm:beligeneral}(i)$ holds for all integral $ \mathcal{O}_{F} $-lattices $ N $ of rank $ n $.
			
(ii)  If $ M $ satisfies  $ I_{1}^{E}(n-1) $ and $ I_{2}^{E}(n-1) $, then Theorem $\ref{thm:beligeneral}(ii)$ holds for all integral $ \mathcal{O}_{F} $-lattices $ N $ of rank $ n $.
		\end{lem}
		\begin{proof}
			 (i) By $ I_{1}^{E}(n-1) $ and Proposition \ref{prop:Ralphaproperty3}(i), we have $ R_{i}=0\le S_{i} $ for odd $ i $ and $ R_{i}=-2e\le S_{i} $ for even $ i $, so  Theorem\;\ref{thm:beligeneral}(i) holds.
			
			 (ii) Similarly, we have $ R_{i+1}=0\le S_{i-1} $ for $ i\in [2,n-1]^{E} $ and $ R_{i+1}=-2e\le S_{i-1}$ for $ i\in [2,n-1]^{O} $. Hence the indices $ 2,\ldots, n-1 $ are not essential. By \cite[Lemma 2.12]{beli_representations_2019}, it remains to prove that $d[a_{1,i}b_{1,i}]\ge A_i$ for $ i=1 $, $n-1 $ and $n $. For $ i=1 $, since $ S_{1}\ge 0 $, $ A_{1}\le (R_{2}-S_{1})/2+e\le (-2e-0)/2+e=0\le d[a_{1}b_{1}]$. For $ i=n-1 $, it follows from Lemma \ref{lem:Aj} with $ j=n-1 $.

			Recall from the definition that $ A_{n}\le (R_{n+1}-S_{n})/2+e $ and $ A_{n}\le R_{n+1}-S_{n}+d[-a_{1,n+1}b_{1,n-1}] $.
			If $ R_{n+1}=-2e $, then since $ S_{n}\ge 0 $, we have $ A_{n}\le (R_{n+1}-S_{n})/2\le (-2e-0)/2+e=0\le d[a_{1,n}b_{1,n}] $. If $ R_{n+1}\not=-2e $, i.e. $ R_{n+1}-R_{n}\not=-2e $, then $ \alpha_{n}\not=0 $ by Proposition \ref{prop:Ralphaproperty2}(i). Hence $ \alpha_{n}=1 $ and $ d[-a_{n}a_{n+1}]=1-R_{n+1} $ by $ I_{2}^{E}(n-1) $. So
				\begin{align*}
			A_{n}\le R_{n+1}-S_{n}+d[-a_{1,n+1}b_{1,n-1}]\le d[a_{1,n}b_{1,n}]
		\end{align*}
		 by Lemma \ref{lem:da1nb1n}.
		\end{proof}
	
		\begin{lem}\label{lem:oddlatticeproperty1}
		Suppose that $ M $ satisfies $ I_{1}^{E}(n-1) $. If $ n=3 $, $ d(a_{1,4})=\infty $, $ R_{4}=-2e $ and $ R_{5}>1 $, then Theorem \ref{thm:beligeneral}(iii) fails at $ i=4 $ when $ N=N_{2}^{3}(\varepsilon\pi) $, with $ \varepsilon\in \mathcal{O}_{F}^{\times} $ (cf. Proposition \ref{prop:maximallattices}).

	\end{lem}
	\begin{proof}
		We have $ R_{5}>1=S_{3}$ by Lemma \ref{lem:maximal-BONG-odd}(ii). Since $ R_{4}=S_{2}=-2e $, it follows that $ d[a_{1,4}]\ge 2e $ and $ d[-b_{1,2}]\ge 2e $ by Proposition \ref{prop:Ralphaproperty3}(iii). Hence $ d[-a_{1,4}b_{1,2}]\ge 2e $ by the domination principle. It follows that
		\begin{align*}
			d[-a_{1,4}b_{1,2}]+d[-a_{1,5}b_{1,3}]\ge 2e+0>2e+1-2\ge 2e+S_{3}-R_{5}\,.
		\end{align*}
	    Since $ R_{4}=-2e $,  by Proposition \ref{prop:Ralphaproperty3}(iv), we have $ [a_{1},a_{2},a_{3},a_{4}]\cong \mathbb{H}^{2}$ or $ \mathbb{H}\perp [1,-\Delta] $. Since also $a_{1,4}\in F^{\times 2} $, we must have $ [a_{1},a_{2},a_{3},a_{4}]\cong \mathbb{H}^{2} $. By definition, $ [b_{1},b_{2},b_{3}]\cong FN_{2}^{3}(\varepsilon\pi)=W_{2}^{3}(\varepsilon\pi) $. But $ \mathbb{H}^{2} $ clearly represents $ W_{1}^{3}(\varepsilon\pi)=\mathbb{H}\perp [\varepsilon\pi]$, so it does not represent $ W_{2}^{3}(\varepsilon\pi) $ by Lemma \ref{lem:spacerep-criterion}.		
	\end{proof}

		\begin{lem}\label{lem:oddlatticeproperty2-1}
		Suppose that $ M $ satisfies $ I_{i}^{E}(n-1) $ for $ i=1,2$. If $ \alpha_{n}=1 $ and either $ R_{n+1}=1 $ or $ R_{n+2}>1 $, then $ d((-1)^{(n+1)/2}a_{1,n+1})=1-R_{n+1} $, $ ((-1)^{(n+1)/2}a_{1,n+1})^{\#} $ is a unit and $ d(((-1)^{(n+1)/2}a_{1,n+1})^{\#})=2e+R_{n+1}-1 $.
	\end{lem}
	
	\begin{proof}
		Since $ \alpha_{n}=1 $,  $ R_{n+1}=R_{n+1}-R_{n}>-2e $ by Proposition \ref{prop:Ralphaproperty2}(i) and $ d[-a_{n}a_{n+1}]=1-R_{n+1} $ by $ I_{2}^{E}(n-1) $. Hence
		\begin{align*}
			1-R_{n+1}=d[(-1)^{(n+1)/2}a_{1,n+1}]\,\in\,\{d((-1)^{(n+1)/2}a_{1,n+1})\,,\,\alpha_{n+1}\}
		\end{align*}
	by Lemma \ref{lem:d[-ai+1ai+2]=1-Ri+2}(ii) and the definition of $ d[(-1)^{(n+1)/2}a_{1,n+1}] $.	From Lemma \ref{lem:d[-ai+1ai+2]=1-Ri+2}(i) we find $  R_{n+1}\in [2-2e,0]^{E}\cup \{1\} $.

If $ R_{n+1}=1 $, then $ d((-1)^{(n+1)/2}a_{1,n+1})=0=1-R_{n+1} $, as $ \ord(a_{1,n+1})$ is odd. Suppose $ R_{n+1}\in [2-2e,0]^{E} $. Then, by the hypothesis, $ R_{n+2}>1 $. If $ R_{n+2}-R_{n+1}>2e $, then $ \alpha_{n+1}>2e>1-R_{n+1} $ by Proposition \ref{prop:Ralphaproperty2}(ii); if $ R_{n+2}-R_{n+1}\le 2e $, then $ \alpha_{n+1}\ge R_{n+2}-R_{n+1}>1-R_{n+1} $ by Proposition \ref{prop:Ralphaproperty2}(iii). In both cases, we see that $ \alpha_{n+1}>1-R_{n+1}=d[(-1)^{(n+1)/2}a_{1,n+1}]  $ and so $ d((-1)^{(n+1)/2}a_{1,n+1})=1-R_{n+1}<2e$.

The other assertions of the lemma follow from Proposition \ref{prop:duality2}.
	\end{proof}

		\begin{lem}\label{lem:oddlatticeproperty2-2}
		Suppose that $ M $ satisfies $ I_{i}^{E}(n-1) $ for $ i=1,2,3 $ (cf. Theorem \ref{thm:even-nuniversaldyadic}). Assume that $ \alpha_{n+2}>G_{n}$ (cf. $\eqref{eq5.1new}$),  $ \alpha_{n}=1 $, and either $ R_{n+1}=1 $ or $ R_{n+2}>1 $. Let $c=(-1)^{(n+1)/2}a_{1,n+2}$ and $\tilde{c}=(-1)^{(n+1)/2}a_{1,n+1} $.

(i) We have $R_{n+2}>S_n$ and $d[-a_{1,n+1}b_{1,n-1}]+d[-a_{1,n+2}b_{1,n}]>2e+S_n-R_{n+2}$ for both $N=N^n_1(c)$ and $N=N^n_1(c\tilde{c}^{\#})$.

(ii) $[a_1,\ldots, a_{n+1}]$ does not represent $FN=[b_1,\ldots, b_n]$ for $N=N^n_1(c)$ or $N=N^n_1(c\tilde{c}^{\#})$.

Thus Theorem \ref{thm:beligeneral}(iii) fails at $ i=n+1 $ for at least one of the lattices  $N_{1}^{n}(c) $ and $ N_{1}^{n}(c\tilde{c}^{\#}) $.
	\end{lem}
	\begin{proof}
(i) Note that $ \ord(a_{1,n})$ is even by $ I_{1}^{E}(n-1) $, and $ \tilde{c}^{\#} $ is a unit by Lemma \ref{lem:oddlatticeproperty2-1}. Hence
\[
\ord(c)\equiv \ord(c\tilde{c}^{\#})\equiv\ord(a_{n+1}a_{n+2})\equiv R_{n+2}-R_{n+1}\pmod{2}\,.
\] So, by Lemma \ref{lem:maximal-BONG-odd}(ii),
\begin{equation}\label{eq5.2new}
S_n=\begin{cases}
  0\quad & \text{ if } R_{n+2}-R_{n+1} \text{ is even}\,,\\
  1\quad & \text{ if } R_{n+2}-R_{n+1} \text{ is odd}\,.
\end{cases}
\end{equation}By the hypothesis and Remark\;\ref{re:Rn+1=Rn+2=1}, either $R_{n+2}=R_{n+1}=1$ or $R_{n+2}>1$.	 If $ R_{n+2}-R_{n+1} $ is even, then $  R_{n+2}\ge 1>S_{n}=0 $. If $ R_{n+2}-R_{n+1} $ is odd, then $  R_{n+2}>1=S_{n} $. We have thus proved that $ R_{n+2}>S_{n}$.

From \eqref{eq5.1new} and \eqref{eq5.2new} we see that
\[
1-R_{n+1}+G_{n}=1-R_{n+1}+(2e-R_{n+2}+R_{n+1}+S_n-1)=2e+S_{n}-R_{n+2}\,.
\]Now we are left to show that
\begin{equation}\label{eq:d-an+1b1n-1d-a1n+2b1n}
			d[-a_{1,n+1}b_{1,n-1}]+d[-a_{1,n+2}b_{1,n}]> (1-R_{n+1})+G_{n} \,.
\end{equation}
		
We have $ N=N_{1}^{n}(c) $ or $ N=N_{1}^{n}(c\tilde{c}^{\#}) $, so in $ F^{\times}/F^{\times 2} $ we have $ b_{1,n}=\det FN=(-1)^{(n-1)/2}c=-a_{1,n+2} $ or $ b_{1,n}=(-1)^{(n-1)/2}c\tilde{c}^{\#}=-a_{1,n+2}\tilde{c}^{\#} $. (Recall that $ c=(-1)^{(n+1)/2}a_{1,n+2} $.) Thus in $ F^{\times}/F^{\times 2} $ we have $ -a_{1,n+2}b_{1,n}=1 $ or $ \tilde{c}^{\#} $, respectively. By Lemma \ref{lem:oddlatticeproperty2-1}, we get
		\begin{align*}	
			&d(-a_{1,n+2}b_{1,n})=
			\begin{cases}
				d(1)=\infty  &\text{if $ N=N_{1}^{n}(c) $}\,,\\
				d(\tilde{c}^{\#})=2e+R_{n+1}-1 &\text{if $ N=N_{1}^{n}(c\tilde{c}^{\#}) $}\,.
			\end{cases}
		\end{align*}
	 	If $ R_{n+2}\le 1 $, then $ R_{n+1}=R_{n+2}=1 $ by Remark \ref{re:Rn+1=Rn+2=1}. So
	 	\begin{align*}
	 			2e+R_{n+1}-1=2e>2e-1=2\left(e-\left\lfloor\dfrac{1-1}{2}\right\rfloor\right )-1=G_{n}\,.
	 	\end{align*}
	 	 If $ R_{n+2}\ge 2 $, then
	 	\begin{align*}
	 			2e+R_{n+1}-1\ge 2e-R_{n+2}+R_{n+1}+1>G_{n}\,.
	 	\end{align*}
	 	 Hence $ d(-a_{1,n+2}b_{1,n})\ge 2e+R_{n+1}-1>G_{n} $. Together with the assumption that $ \alpha_{n+2}>G_{n}  $, this yields
		\begin{align}\label{eq:N3odd1}
			d[-a_{1,n+2}b_{1,n}]=\min\{d(-a_{1,n+2}b_{1,n}),\alpha_{n+2}\}>G_{n}\,.
		\end{align}
	On the other hand, since  $ d[-a_{n}a_{n+1}]=1-R_{n+1} $ (by $ I_{2}^{E}(n-1) $), we have $d[(-1)^{(n+1)/2}a_{1,n+1}]=1-R_{n+1}<2e$ by Lemma \ref{lem:d[-ai+1ai+2]=1-Ri+2}(ii). (We have $ R_{n+1}\ge 2-2e $, so $ 1-R_{n+1}<2e $.) Also, by Lemma\;\ref{lem:maximal-BONG-odd}(ii), we have $S_{n-1}=-2e$, which, by Proposition \ref{prop:Ralphaproperty3}(iii), implies that
	\[
		d[(-1)^{(n-1)/2}b_{1,n-1}] \ge 2e>1-R_{n+1}=d[(-1)^{(n+1)/2}a_{1,n+1}]\,.
	\]Now, by the domination principle, $d[-a_{1,n+1}b_{1,n-1}]=1-R_{n+1}$. From this and \eqref{eq:N3odd1}, we get \eqref{eq:d-an+1b1n-1d-a1n+2b1n} as desired.

(ii)   By Proposition \ref{prop:Ralphaproperty3}(v),
\[
 [a_{1},\ldots, a_{n+1}]\cong [a_{1},\ldots,a_{n}]\perp [a_{n+1}]\cong \mathbb{H}^{(n-1)/2}\perp [(-1)^{(n-1)/2}a_{1,n},a_{n+1}] \,.
  \]By definition, $ [b_{1},\ldots,b_{n}]=FN=\mathbb{H}^{(n-1)/2}\perp [\eta]$, with $\eta\in \{c,\,c\tilde{c}^{\#}\}$.

  Suppose that $\mathbb{H}^{(n-1)/2}\perp [(-1)^{(n-1)/2}a_{1,n},a_{n+1}] $ represents both $ \mathbb{H}^{(n-1)/2}\perp [c] $ and $ \mathbb{H}^{(n-1)/2}\perp [c\tilde{c}^{\#}] $. Then $ [(-1)^{(n-1)/2}a_{1,n},a_{n+1}] $ represents both $ [c] $ and $ [c\tilde{c}^{\#}] $ by Witt cancellation. But $ \det [(-1)^{(n-1)/2}a_{1,n},a_{n+1}]=(-1)^{(n-1)/2}a_{1,n+1}=-\tilde{c} $. It follows that
	   \begin{align*}
	   	[c,-c\tilde{c}]\cong [(-1)^{(n-1)/2}a_{1,n},a_{n+1}] \cong  [c\tilde{c}^{\#},-c\tilde{c}^{\#}\tilde{c}]
	   \end{align*}
   by \cite[63:21 Theorem]{omeara_quadratic_1963}. Scaling by $ c $, we get $ [1,-\tilde{c}]\cong [\tilde{c}^{\#},-\tilde{c}^{\#}\tilde{c}] $. Hence $ \tilde{c}^{\#}\rep [1,-\tilde{c}]$, which implies $ (\tilde{c},\tilde{c}^{\#})_{\mathfrak{p}}=1 $. But this contradicts Proposition \ref{prop:duality2}.
	\end{proof}

		\begin{lem}\label{lem:con3odd}
			Suppose that $ FM $ is $n$-universal and that $ M $ satisfies $ I_{i}^{E}(n-1) $ for $ i=1,\,2,\,3$. Then the following conditions are equivalent:		
			
			\begin{itemize}
				\item[(i)] Theorem $\ref{thm:beligeneral}(iii)$ holds for all integral $ \mathcal{O}_{F} $-lattices $ N $ of rank $ n $.
				\item[(ii)] Theorem $\ref{thm:beligeneral}(iii)$ holds for  the lattices
				\begin{align*}
					N_{1}^{n}(c),\;\;N_{1}^{n}(c\tilde{c}^{\#}),\quad\text{with }\; c=(-1)^{(n+1)/2}a_{1,n+2} \quad\text{and}\quad \tilde{c}=(-1)^{(n+1)/2}a_{1,n+1} \,
				\end{align*} if $ \alpha_{n}=1 $ and either $ R_{n+1}=1 $ or $ R_{n+2}>1$, and for some lattice $N_{2}^{3}(\varepsilon\pi) $, with $\varepsilon\in \mathcal{U} $ (cf. Proposition \ref{prop:maximallattices}), if $ n=3 $,  $ \alpha_{3}=0 $, $ R_{5}>0 $ and $ d(a_{1,4})=\infty $.
				
				\item[(iii)] $ M $ satisfies the condition $ I_{2}^{O}(n) $ in  Theorem $\ref{cor:odd-nuniversaldyadic-simplify}$.
			\end{itemize} 	 	
		\end{lem}
		\begin{proof}
			\textbf{(i)$ \Rightarrow $(ii)}: It is trivial.

			\textbf{(ii)$ \Rightarrow $(iii)}: First assume that $ \alpha_{n}=0 $ and $ R_{n+2}>0 $. Then $ R_{n+1}=R_{n+1}-R_{n}=-2e $ by Proposition \ref{prop:Ralphaproperty2}(i), and thus $ R_{n+2}-R_{n+1}>2e $. By $ I_{3}^{E}(n-1)$, we are left to show that $R_{n+2}=1$ when $ n=3 $ and $ d(a_{1,4})\not=2e$. In this case we have $ n=3 $, $ \alpha_{3}=0 $ and $ R_{5}>0 $. Since $ R_{4}=-2e $, $ d(a_{1,4})\ge d[a_{1,4}]\ge 2e $ by Proposition \ref{prop:Ralphaproperty3}(iii). But $ d(a_{1,4})\not=2e $, so $ d(a_{1,4})=\infty $. If $ R_{5}>1$, then, by Lemma \ref{lem:oddlatticeproperty1}, Theorem \ref{thm:beligeneral}(iii) fails at $ i=4 $ for all lattices $N$ of the form $N=N_{2}^{3}(\varepsilon\pi) $, with $\varepsilon\in \mathcal{O}_{F}^{\times} $. This contradicts condition (ii). Hence $ R_{n+3}=R_{5}=1 $ as desired.
			
			 Now suppose $ \alpha_{n}=1 $ and either $ R_{n+1}=1 $ or $ R_{n+2}>1 $. If $ \alpha_{n+2}>G_{n} $, then, by Lemma \ref{lem:oddlatticeproperty2-2}, Theorem \ref{thm:beligeneral}(iii) fails at $ i=n+1 $ for either $ N=N_{1}^{n}(c) $ or $ N=N_{1}^{n}(c\tilde{c}^{\#}) $. This contradicts condition (ii) again.

			\textbf{(iii)$ \Rightarrow $(i)}: Let $ N^{\prime}=\prec b_{1},\ldots,b_{n-1}\succ$ and $ \beta'_{i}=\alpha_{i}(N^{\prime})$ for $ 1\le i\le n-2 $. A comparison with the definition of $\beta_i=\alpha_i(N)$ (cf. Definition \ref{defn:alpha}) shows that for $ 1\le i\le n-2 $,
			\[
              \beta_{i}=\min\{\beta_{i}^{\prime},S_{n}-S_{i}+d(-b_{n-1}b_{n})\}\le \beta'_i\,.
              \]
			 For $ 0\le i\le m $, $ 0\le j\le n-1 $ and $ c\in F^{\times} $, we denote by $ d^{\prime}[ca_{1,i}b_{1,j}] $ the invariant $ d[ca_{1,i}b_{1,j}] $ corresponding to $ M $ and $ N'$ (cf. \eqref{defn:d[ab]}). Then
			\begin{align*}
				d[ca_{1,i}b_{1,j}] =\min\{d(ca_{1,i}b_{1,j}),\alpha_{i},\beta_{j}\} \quad\text{and}\quad  d^{\prime}[ca_{1,i}b_{1,j}] =\min\{d(ca_{1,i}b_{1,j}),\alpha_{i},\beta_{j}^{\prime}\}\,,
			\end{align*}
			where $ \alpha_{i} $ is ignored if $ i=0 $ or $ m $,  $ \beta_{j} $ and $ \beta_{j}^{\prime} $ are ignored if $ j=0 $, and $ \beta_{j}^{\prime} $ is ignored if $ j=n-1 $.  Since $ \beta_{j}\le \beta_{j}^{\prime} $ for $ 1\le j\le n-2 $, we have $ d[ca_{1,i}b_{1,j}]\le d^{\prime}[ca_{1,i}b_{1,j}] $.
			
			Suppose $ R_{i+1}>S_{i-1} $ and $ d[-a_{1,i}b_{1,i-2}]+d[-a_{1,i+1}b_{1,i-1}]>2e+S_{i-1}-R_{i+1} $ for some $ 2\le i \le n $. Then
			\begin{align*}
				d^{\prime}[-a_{1,i}b_{1,i-2}]+d^{\prime}[-a_{1,i+1}b_{1,i-1}]\ge d[-a_{1,i}b_{1,i-2}]+d[-a_{1,i+1}b_{1,i-1}]>2e+S_{i-1}-R_{i+1}\,.
			\end{align*}
			By Lemma\;\ref{lem:con3even},  Theorem \ref{thm:beligeneral}(iii) holds for $M$ and $ N'$. So  we have $ [b_{1},\ldots,b_{i-1}]\rep [a_{1},\ldots,a_{i}]$.
			
			Hence it suffices to consider Theorem \ref{thm:beligeneral}(iii) for $N$ when $i=n+1$. We assume that $ R_{n+2}>S_{n} $ and $ d[-a_{1,n+1}b_{1,n-1}]+d[-a_{1,n+2}b_{1,n}]>2e+S_{n}-R_{n+2} $, and we want to show that $FN=[b_1,\ldots, b_n]$ is represented by $[a_1,\ldots, a_{n+1}]$.

Recall that, by $I^E_2(n-1)$, we have $\alpha_n=0$ or $\alpha_n=1$.
			
			\textbf{Case I: $ \alpha_{n}=0 $.}
			
			 We have  $ R_{n+2}\le 1 $ by $ I_{2}^{O}(n) $. Also, $ S_{n}\ge 0 $. Then the assumption $ R_{n+2}>S_{n} $ implies $ R_{n+2}=1 $ and $ S_{n}=0 $. Since $ R_{n+1}=-2e $,  $ \ord(a_{1,n+1})$ is even by Proposition \ref{prop:Ralphaproperty3}(iii). Since $ S_{n}=0 $, $ \ord( b_{1,n}) $ is even by Proposition \ref{prop:Ralphaproperty3}(ii). It follows that $\ord(a_{1,n+2}b_{1,n})$ is odd and thus $ d[-a_{1,n+2}b_{1,n}]=0 $. Hence
			\begin{align*}
				\beta_{n-1}= \beta_{n-1}+0\ge d[-a_{1,n+1}b_{1,n-1}]+d[-a_{1,n+2}b_{1,n}]>2e+S_{n}-R_{n+2}=2e-1\,.
			\end{align*} So $ \beta_{n-1}\ge 2e $ and hence $ S_{n}-S_{n-1}\ge 2e$ by Proposition \ref{prop:Ralphaproperty2}(ii). However, $ S_{n}=0 $ and $ S_{n-1}\ge -2e $ (by Proposition \ref{prop:Ralphaproperty3}(i)), so $ S_{n-1}=-2e $. Since $ S_{n}=S_{n-1}+2e=0 $, $ [b_{1},\ldots,b_{n}]\cong W_{1}^{n}(\varepsilon) $ for some $ \varepsilon\in \mathcal{O}_{F}^{\times} $ by Proposition \ref{prop:Ralphaproperty3}(v). Recall that $ R_{n+1}=-2e $, so $ [a_{1},\ldots,a_{n+1}]\cong W_{1}^{n+1}(1) $ or $ W_{1}^{n+1}(\Delta) $ by Proposition \ref{prop:Ralphaproperty3}(iv). In both cases, $ [b_{1},\ldots,b_{n}]\rep [a_{1},\ldots,a_{n+1}] $ by Lemma \ref{lem:spacerep-criterion-2}(ii).
			
			 \textbf{Case II: $ \alpha_{n}=1 $.}

		    By Lemma \ref{lem:da1nb1n}, we have $ R_{n+1}-S_{n}+d[-a_{1,n+1}b_{1,n-1}]\le d[a_{1,n}b_{1,n}]\le \alpha_{n}=1 $ and so
		    \begin{align}\label{d-a1n+1b1n-1le}
		    	d[-a_{1,n+1}b_{1,n-1}] \le S_{n}-R_{n+1}+d[a_{1,n}b_{1,n}]\le S_{n}-R_{n+1}+1\,.
		    \end{align}
	    Assume that $ d[-a_{1,n+2}b_{1,n}]=0 $. Then
		    \[
		    	d[-a_{1,n+1}b_{1,n-1}]=d[-a_{1,n+1}b_{1,n-1}]+d[-a_{1,n+2}b_{1,n}]>2e+S_{n}-R_{n+2}\,.
		    \]This combined with \eqref{d-a1n+1b1n-1le} shows that $ R_{n+2}-R_{n+1}>2e-1 $, contradicting Lemma \ref{lem:simplifythm}(ii). So we must have $ d[-a_{1,n+2}b_{1,n}]>0 $.

			If $ R_{n+1}\in [2-2e,0]^{E} $ and $ R_{n+2}=1 $, then $ S_{n}=0 $ since $ R_{n+2}>S_{n}\ge 0$.  Thus $\ord(b_{1,n})$ is even by Proposition \ref{prop:Ralphaproperty3}(ii). Note that $ \ord( a_{1,n+2})$ is odd, so  $ \ord(a_{1,n+2}b_{1,n})$ is odd, which implies  $ d[-a_{1,n+2}b_{1,n}]=0 $, a contradiction. So we have either $R_{n+1}=1$ or $R_{n+2}>1$. (Recall that $R_{n+2}>S_n\ge 0$.) Then $ d[-a_{1,n+2}b_{1,n}]\le \alpha_{n+2}\le G_{n}  $ by
$ I_{2}^{O}(n) $. Combining this with the first inequality in $ \eqref{d-a1n+1b1n-1le} $, we see that
	\[
				S_{n}-R_{n+1}+d[a_{1,n}b_{1,n}]+G_n\ge  d[-a_{1,n+1}b_{1,n-1}]+d[-a_{1,n+2}b_{1,n}]>2e+S_{n}-R_{n+2}\,.
	\]It follows that
\[
1=\alpha_n\ge d[a_{1,n}b_{1,n}]>2e-R_{n+2}+R_{n+1}-G_n\,.
\]From \eqref{eq5.1new} we see that $2e-R_{n+2}+R_{n+1}-G_n\in \{0,\,1\}$. It follows that $ d[a_{1,n}b_{1,n}]=1 $, which implies that $ \ord(a_{1,n}b_{1,n}) $ is even, and $ 2e-R_{n+2}+R_{n+1}-G_{n}=0 $, which implies, by \eqref{eq5.1new}, that $ \ord(a_{n+1}a_{n+2})=R_{n+1}+R_{n+2} $ is odd. Hence we deduce that $ \ord(a_{1,n+2}b_{1,n}) $ is odd, so $ d[-a_{1,n+2}b_{1,n}]=0 $, a contradiction again.
 \end{proof}

		\begin{lem}\label{lem:con4odd}
			Suppose that $ FM $ is $n$-universal and that $ M $ satisfies $ I_{i}^{E}(n-1) $ for $ i=1,\,2,3 $ and $I_{2}^{O}(n)$. Then the following conditions are equivalent:
			
			\begin{itemize}
				\item[(i)] Theorem $\ref{thm:beligeneral}(iv)$ holds for all integral $ \mathcal{O}_{F} $-lattices $ N $ of rank $ n $.
				
				\item[(ii)] Theorem $\ref{thm:beligeneral}(iv)$ holds for the lattices $ N_{1}^{n}(c)$ and $ N_{2}^{n}(c) $, with $ c=(-1)^{(n+1)/2}a_{1,n+2}  $.
				\item[(iii)] $ M $ satisfies the condition $ I_{3}^{O}(n) $, i.e., $R_{n+3}-R_{n+2}\le 2e$.
			\end{itemize}
		\end{lem}
		\begin{proof}

			\textbf{(i)$ \Rightarrow $(ii)}: It is trivial.
			
			\textbf{(ii)$ \Rightarrow $(iii)}: Suppose $ R_{n+3}-R_{n+2}>2e $. We claim that $ R_{n+2}=0 $ or $ 1 $ and that $ R_{n+1} $ is even.
			
			By $ I_{2}^{E}(n-1) $, we have $ \alpha_{n}=0 $ or $ 1 $. If $ \alpha_{n}=0 $, then we have $ R_{n+1}=-2e $ by Lemma \ref{lem:simplifythm}(i), and $ R_{n+2}=0 $ or $ 1 $ by $I_{2}^{O}(n)$. Thus the claim is true in this case.

		 If $ \alpha_{n}=1 $, then $R_{n+1}=R_{n+1}-R_{n}\in [2-2e,0]^{E} \cup\{1\} $ by Proposition \ref{prop:Ralphaproperty2}(vi). If the claim were not true, we would have either $ R_{n+1}=1 $ or $ R_{n+2}>1 $. Then $ R_{n+2}\ge 1\ge R_{n+1} $ by Remark \ref{re:Rn+1=Rn+2=1}, and thus $ R_{n+2}-R_{n+1}\ge 0 $. By $ I_{2}^{O}(n) $, we get
			\[
				\alpha_{n+2}\le G_n=2\bigg(e-\bigg\lfloor\dfrac{R_{n+2}-R_{n+1}}{2}\bigg\rfloor\bigg)-1\le 2\bigg(e-\bigg\lfloor\dfrac{0}{2}\bigg\rfloor\bigg)-1=2e-1\,.
			\] So $ R_{n+3}-R_{n+2}\le 2e $ by Proposition \ref{prop:Ralphaproperty2}(ii), a contradiction. Our claim is thus proved.
			
			Take $ N=N_{\nu}^{n}(c) $, with $ \nu\in \{1,2\} $ and $ c=(-1)^{(n+1)/2}a_{1,n+2} $. By $ I_{1}^{E}(n-1) $ and the claim, $ R_{i} $ is even for $ 1\le i\le n+1 $, so $ \ord(c)=\ord( a_{1,n+2})\equiv R_{n+2}\pmod{2}$. If $ R_{n+2}=0 $, then $ \ord(c)\equiv 0\pmod{2} $ and so $ N=N_{\nu}^{n}(\varepsilon) $ for some
$\varepsilon\in \mathcal{O}_{F}^{\times} $. Hence $ S_{n}=0=R_{n+2} $ by Lemma \ref{lem:maximal-BONG-odd}(ii). If $ R_{n+2}=1 $, then similarly $ N=N_{\nu}^{n}(\varepsilon\pi) $ for some $ \varepsilon\in \mathcal{O}_{F}^{\times} $. Hence $ S_{n}=1=R_{n+2} $ by Lemma \ref{lem:maximal-BONG-odd}(ii).
			
			Now we have  $ S_{n}=R_{n+2} $ and $ R_{n+3}-R_{n+2}>2e $. So the condition $ R_{n+3}>S_{n}+2e\ge R_{n+2}+2e$ is satisfied. Therefore,
$ W^n_{\nu}(c)=FN=[b_{1},\ldots,b_{n}]$ is represented by $V:=[a_{1},\ldots,a_{n+2}]  $ by Theorem\;\ref{thm:beligeneral}(iv). Note that $ \det V=a_{1,n+2}=(-1)^{(n+1)/2}c=-\det FN$. So $ V $ represents exactly one of $ W_{1}^{n}(c) $ and $ W_{2}^{n}(c) $ by Lemma \ref{lem:spacerep-criterion}. A contradiction is derived and so $ R_{n+3}-R_{n+2}\le 2e $.

			\textbf{(iii)$ \Rightarrow $(i):} For $ 1<i\le n-1 $, $ R_{i+1}-R_{i}\le 2e $ by $ I_{1}^{E}(n-1) $. Since $ \alpha_{n}\le 1 $ by $ I_{2}^{E}(n-1) $, $ R_{n+1}-R_{n}\le 2e $ by Proposition \ref{prop:Ralphaproperty2}(ii). Also, $ R_{n+3}-R_{n+2}\le 2e $ by $ I_{3}^{O}(n) $.	Hence we only need to consider Theorem \ref{thm:beligeneral}(iv) for $ i=n $. Assume that $ S_{n}\ge R_{n+2}>S_{n-1}+2e\ge R_{n+1}+2e $. Since $ M $ satisfies $ I_{i}^{E}(n-1) $ for $ i=1,2,3 $, it is $ (n-1) $-universal by Theorem \ref{thm:even-nuniversaldyadic} and thus represents the integral lattice $ N^{\prime}:=\prec b_{1},\ldots,b_{n-1}\succ $. Since $ R_{n+2}>S_{n-1}+2e=R_{n+1}+2e $, we see that $ [b_{1},\ldots,b_{n-1}]\rep [a_{1},\ldots,a_{n+1}] $ by applying Theorem \ref{thm:beligeneral}(iv) to $ M $ and $ N^{\prime} $.
		\end{proof}

		\section{Proof of the criterion for $ n $-universality}\label{sec:proof-main}
		
		The statements of Theorems \ref{thm:even-nuniversaldyadic} and \ref{cor:odd-nuniversaldyadic-simplify} involve not only the $R$-invariants but also the $\alpha$-invariants.  In this section, we prove Theorem \ref{thm:nuniversaldyadic}, where necessary and sufficient conditions for the $n$-universal property are given without using  the $\alpha$-invariants explicitly.

		As before, let $M$ be an integral $\mathcal{O}_F$-lattice and suppose that $M\cong \prec a_1,\ldots, a_m \succ$ relative to some good BONG.

		\begin{lem}\label{lem:I2E}
			Let $ n\ge 2 $ be an even integer. Suppose that $M$ satisfies $ I_{1}^{E}(n) $.
			
			(i) Theorem \ref{thm:nuniversaldyadic}(ii)(1)(a) holds if and only if $M$ satisfies $ I_{2}^{E}(n) $ when $ R_{n+2}=2-2e $.
			
			(ii) Theorem \ref{thm:nuniversaldyadic}(ii)(1)(b) holds if and only if $M$ satisfies $ I_{2}^{E}(n) $ when $ R_{n+2}\not=2-2e $.
		\end{lem}
	
	   \begin{proof}
	   	By $ I_{1}^{E}(n) $, we have $ R_{i}=0$ for $\in [1,n+1]^{O} $ and $ R_{i}=-2e $ for $ i\in [1,n+1]^{E} $.
	   	
	   (i) Assume that $ R_{n+2}=2-2e $, i.e. $ R_{n+2}-R_{n+1}=2-2e $. Then $ \alpha_{n+1}=1 $ by Proposition \ref{prop:Ralphaproperty2}(iv). Hence $ I_{2}^{E}(n) $ holds if and only if $ d[-a_{n+1}a_{n+2}]=1-R_{n+2}=2e-1 $.
	We claim that
	     \begin{equation}\label{claim6.1}
	     	\alpha_{n+2}=2e-1\quad\text{if and only if}\quad R_{n+3}\in \{0,1\}\,.
	     \end{equation}
	       If $ R_{n+3}\in \{0,1\} $, then $ R_{n+3}-R_{n+2}\in \{2e-2,2e-1\}$, so $ \alpha_{n+2}=2e-1 $ by Proposition \ref{prop:Ralphaproperty2}(iv) and (iii). Conversely, if $ \alpha_{n+2}=2e-1 $, then $ R_{n+3} <2e+R_{n+2}=2 $ by Proposition \ref{prop:Ralphaproperty2}(ii). Since $ n+3 $ is odd, $ R_{n+3}\ge 0 $ and thus $ R_{n+3}\in \{0,1\} $. The claim is proved.

	     Since $ R_{n+1}-R_{n}=2e $, $ \alpha_{n}=2e $ by Proposition \ref{prop:Ralphaproperty2}(ii). By the definition of $ d[-a_{n+1}a_{n+2}] $ and Lemma \ref{lem:d[-ai+1ai+2]=1-Ri+2}(i), we have
	     \begin{align*}
	     	\min\{d(-a_{n+1}a_{n+2}),2e,\alpha_{n+2}\}&=	\min\{d(-a_{n+1}a_{n+2}),\alpha_{n},\alpha_{n+2}\}\\
	     	&=d[-a_{n+1}a_{n+2}]\ge 1-R_{n+2}=2e-1\,.
	     \end{align*}
	     Hence, in the present case we obtain
\[
	     \begin{split}
I^E_2(n)&\iff	     	d[-a_{n+1}a_{n+2}]=2e-1\\
&\iff d(-a_{n+1}a_{n+2})=2e-1\quad\text{or}\quad \alpha_{n+2}=2e-1 \\
	     \text{by \eqref{claim6.1}}\;\;&\iff d(-a_{n+1}a_{n+2})=2e-1\quad\text{or}\quad R_{n+3}\in \{0,1\}\\
&\iff \text{Theorem \ref{thm:nuniversaldyadic}(ii)(1)(a)}\;.
	     \end{split}
\]

(ii) Assume that $R_{n+2}\neq 2-2e$. If $ R_{n+2}=1 $, i.e. $ R_{n+2}-R_{n+1}=1$, then $ \alpha_{n+1}=1 $ by Proposition \ref{prop:Ralphaproperty2}(iii). Moreover, $d[-a_{n+1}a_{n+2}]=0=1-R_{n+2}$ since $\ord(a_{n+1}a_{n+2})$ is odd. If $R_{n+2}\in [4-2e,\,0]^E$, then $ R_{n+2}-R_{n+1}\in [4-2e,0]^{E}  $. Note that $ \alpha_{n+1}=1 $ is equivalent to $ d[-a_{n+1}a_{n+2}]=1-R_{n+2} $ by Proposition \ref{prop:Ralphaproperty2}(vii). Hence $ I_{2}^{E}(n) $ holds if and only if $ \alpha_{n+1}=1 $.

       Recall Definition \ref{defn:alpha} and write $ T_{j}=T_{j}^{(n+1)} $ for $ 0\le j\le m-1 $ for short. Then $ \alpha_{n+1}=\min\{T_{0},\ldots,T_{m-1}\} $.
    Since $ R_{n+2}-R_{n+1}=R_{n+2}>-2e $, $ \alpha_{n+1}\ge 1 $ by Proposition \ref{prop:Ralphaproperty2}(i). Hence $ \alpha_{n+1}=1 $ if and only if $ 1\in \{T_{0},\ldots,T_{m-1}\} $.

 Since $ R_{n+2}\ge 4-2e $, we have $ T_{0}=(R_{n+2}-R_{n+1})/2+e\ge (4-2e)/2+e=2>1  $. For $j\in [1,n]^{O} $, $ R_{j}=0 $ and $ d(-a_{j}a_{j+1})\ge 2e$ (by Proposition \ref{prop:Ralphaproperty3}(iii)); for $ j\in [1,n]^{E} $, $ R_{j}=-2e $ and $ d(-a_{j}a_{j+1})\ge 0 $. Hence for all $1\le j\le n$ we have
 \begin{align}\label{Rj+dajaj+1}
 	 -R_{j}+d(-a_{j}a_{j+1})\ge 2e\,,
 \end{align}and so  $ T_{j}= R_{n+2}-R_{j}+d(-a_{j}a_{j+1})\ge (4-2e)+2e=4>1 $. So $ \alpha_{n+1}=1 $ if and only if $ 1\in \{T_{n+1},\ldots,T_{m-1}\} $, i.e. there exists some $j $ with $ n+1\le j\le m-1 $ for which
 \begin{align*}
 	1=T_{j}=R_{j+1}-R_{n+1}+d(-a_{j}a_{j+1})=R_{j+1}+d(-a_{j}a_{j+1})\,,
 \end{align*}
 that is, $ d(-a_{j}a_{j+1})=1-R_{j+1}$. This is just Theorem \ref{thm:nuniversaldyadic}(ii)(1)(b).
\end{proof}
		
\begin{cor}\label{cor:thm1.1equiv}
	If $n\ge 3$ is odd, then Theorem \ref{thm:nuniversaldyadic}(i) and (iii)(1) hold if and only if $I_{1}^{O}(n)$ holds.
\end{cor}		
\begin{proof}
We may assume that $R_i=0$ for all $i\in [1,\,n]^O$ and $R_i=-2e$ for all $i\in [1,\,n]^E$. Under this assumption, we need to show that Theorem \ref{thm:nuniversaldyadic}(iii)(1) holds if and only if $\alpha_n=0$ or 1. Since $R_n=0$, we have $R_{n+1}-R_{n}=R_{n+1}$.

By Proposition \ref{prop:Ralphaproperty2}(i), we have $\alpha_{n}=0$ if and only if $R_{n+1}=R_{n+1}-R_{n}=-2e$. If $\alpha_n=1$, then, by Proposition \ref{prop:Ralphaproperty2}(vi), we have $R_{n+1}=R_{n+1}-R_{n}\in [2-2e,0]^{E}\cup\{1\}$. Conversely, if $R_{n+1}-R_{n}=2-2e$ or $1$, then $\alpha_{n}=1$ by Proposition \ref{prop:Ralphaproperty2}(iv) and (iii) respectively.

It remains to consider the case $R_{n+1}=R_{n+1}-R_{n}\in [4-2e,0]^{E}$. In this case, $\alpha_{n}=1$ if and only if $d[-a_{n}a_{n+1}]=1-R_{n+1}$ by Proposition \ref{prop:Ralphaproperty2}(vii). Hence,  by Lemma \ref{lem:I2E}(ii), the conditions $\alpha_{n}=1$ and Theorem \ref{thm:nuniversaldyadic}(iii)(1) are equivalent.
\end{proof}

		\begin{lem}\label{lem:I2O}
			Let $ n\ge 3 $ be an odd integer. Suppose that $M$ satisfies $ I_{1}^{O}(n) $. Then Theorem \ref{thm:nuniversaldyadic}(iii)(2) holds if and only if the second part of $ I_{2}^{O}(n) $ holds.
		\end{lem}
	    \begin{proof}Note that $R_n=0$ and $\alpha_n\in\{0,\,1\}$ by $I_{1}^{O}(n)$. Thus $\alpha_n=1$ if and only if $R_{n+1}\neq -2e$.
	    	The second part of $ I_{2}^{O}(n) $ states that if $ R_{n+1}\not=-2e $ (i.e. $ \alpha_{n}=1 $) and either $ R_{n+1}=1 $ or $ R_{n+2}>1 $, then $ \alpha_{n+2}\le G_{n} $ (cf. \eqref{eq5.1new}).   We must show that this statement is equivalent to the inequalities in Theorem \ref{thm:nuniversaldyadic}(iii)(2).

Recall Definition \ref{defn:alpha} and write $ T_{j}=T_{j}^{(n+2)} $ for $ 0\le j\le m-1 $ for brevity. Then $ \alpha_{n+2} =\min\{T_{0},\ldots,T_{m-1}\}$. By Proposition \ref{prop:Ralphaproperty3}(i), $ R_{n+3}\ge R_{n+1}\ge -2e $, and by Remark \ref{re:Rn+1=Rn+2=1}, $ R_{n+2}\ge 1 $. Hence
	    	\begin{align*}
	    		R_{n+2}+R_{n+3}\ge R_{n+2}+R_{n+1}\ge 1+(-2e)>-2e\,,
	    	\end{align*}
	    	which is equivalent to the inequality
	    	\[
	    		R_{n+3}+2e>\dfrac{R_{n+3}-R_{n+2}}{2}+e=T_{0}\,.
	    	\]
	    	 For $ 1\le j\le n-1 $, by $ I_{1}^{O}(n) $ and Proposition \ref{prop:Ralphaproperty3}(iii), we have $ -R_{j}+d(-a_{j}a_{j+1})\ge 2e $ (cf. \eqref{Rj+dajaj+1}) and so
	    	\begin{align*}
	    		T_{j}=R_{n+3}-R_{j}+d(-a_{j}a_{j+1})\ge R_{n+3}+2e>T_{0}\,.
	    	\end{align*}
    		Hence $ \alpha_{n+2}=\min\{T_{0},T_{n}\ldots,T_{m-1}\} $.
    	   For $ j=n,n+1 $, we claim that $ T_{j}+G_{n}\ge 2T_{0} $. By \eqref{eq5.1new}, $t:=2e-R_{n+2}+R_{n+1}-G_n\in \{0,1\}$. We have
	    \begin{align*}
	    	T_{j}+G_{n}&=(R_{n+3}-R_{j}+d(-a_{j}a_{j+1}))+(2e-R_{n+2}+R_{n+1}-t)\\
	    	&=(R_{n+3}-R_{n+2}+2e)+(R_{n+1}-R_{j}+d(-a_{j}a_{j+1})-t)\\
	    	&=2T_{0}+R_{n+1}-R_{j}+d(-a_{j}a_{j+1})-t\,.
	    \end{align*}
	   It suffices to verify that $ R_{n+1}-R_{j}+d(-a_{j}a_{j+1})-t\ge 0 $. If $ j=n $, then $ R_{n+1}-R_{n}+d(-a_{n}a_{n+1})-t\ge \alpha_{n}-t=1-t\ge 0 $ by \eqref{eq:alpha-defn}. Let now $j=n+1$. If $R_{n+2}-R_{n+1}$ is even, then $d(-a_{n+1}a_{n+2})\ge 1=t$. If $R_{n+2}-R_{n+1}$ is odd, then $d(-a_{n+1}a_{n+2})=t=0$. In both cases we have $ R_{n+1}-R_{n+1}+d(-a_{n+1}a_{n+2})-t=d(-a_{n+1}a_{n+2})-t\ge 0 $, as required. Thus the claim is proved.
	
	   We have $ \alpha_{n+2}=\min\{T_{0},T_{n},\ldots,T_{m-1}\}\le G_{n} $ if and only if $ T_{k}\le G_{n} $ for some $k\in \{0,n,\ldots,m-1\} $. But for $ j=n $ or $ n+1 $, if $ T_{j}\le G_{n}$, then, by the claim,  $ 2T_{0}\le T_{j}+G_{n}\le 2G_{n}$, i.e. $ T_{0}\le G_{n} $. Hence $ \alpha_{n+2}\le G_{n} $ if and only if $ T_{k}\le G_{n} $ for some $ k\in \{0,n+2,\ldots,m-1\} $.
	
	   Now, one can verify that
	   \begin{align*}
	   	T_{0}\le G_{n}\quad&\Longleftrightarrow\quad\dfrac{R_{n+3}-R_{n+2}}{2}+e\le 2e-R_{n+2}+R_{n+1}-t\\
	   	&\Longleftrightarrow\quad R_{n+3}+R_{n+2}-2R_{n+1}\le 2e-2t
	   \end{align*}
	   and for $n+2\le j\le m-1$,
	   \begin{align*}
	   	T_{j}\le G_{n}\quad& \Longleftrightarrow\quad R_{j+1}-R_{n+2}+d(-a_{j}a_{j+1})\le 2e-R_{n+2}+R_{n+1}-t\\
	   	& \Longleftrightarrow\quad d(-a_{j}a_{j+1})\le 2e+R_{n+1}-R_{j+1}-t\,.
	   \end{align*} Since $ t=1 $ or $ 0 $ accordingly as $ R_{n+2}-R_{n+1} $ is even or odd, these inequalities coincide with those in Theorem \ref{thm:nuniversaldyadic}(iii)(2).	
	    \end{proof}

		\begin{proof}[Proof of Theorem \ref{thm:nuniversaldyadic}]
			By \cite[Theorem\;2.3]{hhx_indefinite_2021}, $ FM $ is $ n $-universal if and only if either $ m=n+2=4 $ and $ FM\cong \mathbb{H}^{2} $, or $ m\ge n+3 $. Clearly, we may assume that this condition holds.
			
		The case $ m=n+2=4 $ has been treated in Corollary \ref{cor:even-nuniversaldyadic-quaternary}. Now suppose $ m\ge n+3 $. For even $ n\ge 2 $, we have clearly
			\[
			\text{(i) and $ R_{n+1}=0 $}\Longleftrightarrow I_{1}^{E}(n)\,;\quad \text{(ii)(2)}\Longleftrightarrow I_{3}^{E}(n)\,.
			\]Hence we are done by Theorem \ref{thm:even-nuniversaldyadic} and Lemma \ref{lem:I2E}.
			
			Suppose that $ n\ge 3$ is odd. The condition (iii)(4) is the same as  $I_{3}^{O}(n)$. Note that $ \alpha_{n}=0 $ is equivalent to $ R_{n+1}=R_{n+1}-R_{n}=-2e $ by Proposition \ref{prop:Ralphaproperty2}(i). Hence the condition (iii)(3) is equivalent to the first statement of $ I_{2}^{O}(n) $. By Corollary \ref{cor:thm1.1equiv}, the conditions (i) and (iii)(1) are equivalent to $I_{1}^{O}(n)$.
			
			Assume that $I_{1}^{O}(n)$ holds. Then the condition (iii)(2) is equivalent to the second statement of $I_{2}^{O}(n)$ by Lemma \ref{lem:I2O}. Hence we are done by Theorem \ref{thm:odd-nuniversaldyadic-2}.	
		\end{proof}

		\section*{Acknowledgments}
		We are indebted to the referee for very carefully reading the manuscript and for giving many comments and suggestions, which significantly improved the exposition of the paper.		We thank Prof.\,Fei Xu for helpful discussions. This work was supported by a grant from the National Natural Science Foundation of China (No.\,12171223) and the Guangdong Basic and Applied Basic Research Foundation (No.\,2021A1515010396).

			\end{document}